\documentclass[aps,pre,reprint,superscriptaddress,floatfix]{revtex4-2}

\usepackage{graphicx}
\usepackage{amsmath,amssymb,amsthm}
\usepackage{bm}
\usepackage{booktabs}
\usepackage{placeins}
\usepackage{enumitem}
\usepackage{tikz}
\usepackage{pgfplots}
\pgfplotsset{compat=1.18}
\usepgfplotslibrary{groupplots}

\pgfplotsset{
  papel/.style={
    width=\columnwidth, height=0.52\linewidth,
    every axis plot post/.append style={line width=1pt},
    tick label style={font=\small}, label style={font=\small},
    legend style={font=\footnotesize, draw=black!30},
    axis line style={black!60}, tick align=outside,
  }
}

\theoremstyle{plain}
\newtheorem{theorem}{Theorem}[section]
\newtheorem{lemma}[theorem]{Lemma}
\newtheorem{proposition}[theorem]{Proposition}
\newtheorem{corollary}[theorem]{Corollary}
\theoremstyle{definition}
\newtheorem{definition}[theorem]{Definition}
\theoremstyle{remark}
\newtheorem{remark}[theorem]{Remark}

\makeatletter
\AtBeginDocument{%
  \setlength{\@fptop}{0pt}
  \setlength{\@fpsep}{14pt}
  \setlength{\@dblfptop}{0pt}
  \setlength{\@dblfpsep}{14pt}%
  \setlength{\@fpbot}{0pt plus 1fil}%
  \setlength{\@dblfpbot}{0pt plus 1fil}%
}
\makeatother
\begin{document}

\title{Channel Duality and Midpoint Invariance in Random Walks with Resetting}
\author{J. A. Vega Coso}
\affiliation{Instituto Universitario de F\'isica Fundamental y Matem\'aticas
(IUFFyM),\\ Universidad de Salamanca, E-37008 Salamanca, Spain}
\begin{abstract}
We study absorption of a biased nearest-neighbor random walk on a
finite interval that is reset to its initial position. The central
contribution is an exact duality between the two absorption channels,
which holds mode by mode in the spectral representation and step by
step in time, together with the invariance it produces: when the
interval has even length, the ruin probability from its midpoint
equals the classical reset-free value for every law of the reset times
independent of the walk, provided absorption occurs almost surely.
Away from the midpoint, the same duality shows that every such
resetting mechanism favors the nearer boundary. For geometric
resetting we obtain the ruin probability in closed form and prove that
its response to the resetting probability changes sign at the midpoint
and nowhere else, for every bias; when the interval has odd length the
sign change falls between the two central sites. For a diffusion with
drift the midpoint invariance holds for every resetting mechanism, and
under Poissonian resetting the sign change is again unique. Channel
duality is thus the mechanism behind the invariance and behind its
robustness to the resetting protocol and to the passage to the
continuum. All analytic results are checked against an independent
linear solve and against simulation.
\end{abstract}

\maketitle

\section{Introduction}
\label{sec:introduction}

\subsection{Random walks, absorption, and classical ruin theory}

Random walks constitute one of the cornerstones of probability theory and
statistical physics, providing a universal framework for modeling
transport, diffusion, and stochastic exploration. Since Pearson's early
formulation of random flight \cite{Pearson1905} and Einstein's theory of
Brownian motion \cite{Einstein1905}, random walk models have proven
indispensable across disciplines, from polymer conformations
\cite{DeGennes1979} and animal foraging \cite{Viswanathan1996} to
financial fluctuations \cite{Mantegna2000} and the study of random walks on graphs
\cite{Lovasz1993}; first-passage and absorption problems have played a
central role throughout \cite{Redner2001, Condamin2007}.

Among the most classical examples is the \emph{gambler's ruin problem},
given its modern treatment by Feller \cite{Feller1968}.
Consider a biased random walk on $\{0,1,\dots,a\}$ with absorbing
boundaries at $0$ and $a$, stepping right with probability $p$ and left
with probability $q=1-p$. Starting from an initial position $z$, one asks
for the probability $q_z$ of eventual absorption at $0$ before reaching
$a$. The classical solution depends exponentially on the bias ratio
$q/p$ and on the initial position, and has found interpretations ranging
from finance to particle absorption in confined geometries. Absorption
probabilities such as $q_z$ quantify not merely \emph{when} a stochastic
process terminates, but \emph{where}; in finite domains they encode the
geometric interplay between bias, confinement, and initial condition.

\subsection{First passage under resetting}

Stochastic resetting has emerged as a powerful mechanism for controlling
random processes. The seminal work of Evans and Majumdar \cite{Evans2011} showed that diffusion with Poissonian resetting reaches a non-equilibrium stationary state with a finite, optimizable mean first-passage time. This work initiated an extensive body of research on stochastic resetting, including time-dependent resetting protocols \cite{Pal2016} and extensions to L\'evy processes, search problems, and nonequilibrium systems more broadly \cite{Evans2020}. A consistent theme in this literature is the optimization of
\emph{temporal} observables. Absorption probabilities in confined
domains have also been studied: for Brownian motion in an interval
with two absorbing endpoints, splitting probabilities under resetting
have been obtained~\cite{PalPrasad2019, Villarroel2022, Abundo2025}.
The present author has worked on two-barrier absorption problems in a
series of collaborations: on a semi-deterministic walk with
resetting~\cite{Villarroel2021}, and, without resetting, on escape
probabilities for compound Poisson processes with
drift~\cite{Villarroel2023a} and for renewal processes with two-sided
jumps~\cite{Villarroel2023b}. What has not been brought to bear on
these problems is a \emph{spectral} treatment of the two absorption
channels separately, and it is that treatment---rather than any single
observable---that this paper introduces.

In discrete time, the relevant literature is more recent and more
directly comparable to the present problem. Ku\'smierz \emph{et al.}
\cite{Kusmierz2014} studied discrete-time random
search with resetting, deriving the mean first-passage time to a single
target for arbitrary jump distributions. Riascos \emph{et al.}
\cite{Riascos2020} developed a spectral theory of resetting on arbitrary
networks, obtaining the stationary distribution and first-passage
statistics from the spectrum of the transition operator; the same
group later used this framework to derive a general criterion for when
resetting reduces the mean first-passage time, for both random walks and
L\'evy flights on networks \cite{Riascos2022}. Bonomo and Pal
\cite{Bonomo2021} built a general renewal framework for first-passage
under discrete-time restart in discrete space, applied it to symmetric
and biased lattice walks confined by absorbing boundaries, and obtained
both the full first-passage-time distribution and the criterion for when
restart is advantageous. Chen, Li, and Huang \cite{ChenLiHuang2022}
derived, for discrete-time absorbing Markov chains under resetting,
exact expressions for the splitting probabilities and for the
conditional and unconditional mean first-passage times, through a
deformed fundamental matrix. Das and Giuggioli \cite{DasGiuggioli2022} developed a
discrete-space-time renewal framework giving closed-form expressions
for the resetting dynamics of lattice walks in terms of the reset-free
propagator, and applied it to biased walks in unbounded space and in
bounded domains with periodic and reflecting boundaries.

Taken together, these works supply general renewal and spectral
frameworks for discrete-time resetting and give splitting
probabilities for broad classes of reset processes, including
absorbing Markov chains and confined lattice walks. The ruin
probability studied here falls within their scope: it is the splitting
probability of Ref.~\cite{ChenLiHuang2022} specialized to a
nearest-neighbor walk reset to its starting site, and the setting of
Ref.~\cite{Bonomo2021} --- a biased lattice walk between two absorbing
boundaries, reset to its initial site --- is closely related to ours.
The one-dimensional splitting example of Ref.~\cite{ChenLiHuang2022}
is the symmetric walk, where reflection symmetry fixes the midpoint
value at $1/2$ for every resetting rate.

What we add is structural rather than a further formula. We make
explicit an exact duality between the two absorption channels of the
biased walk, visible mode by mode in the spectral representation and,
equivalently, step by step in the reset-free absorption laws. At the
reflection-fixed point it makes the two channels proportional at every
time; consequently, the midpoint ruin probability equals its classical
reset-free value for every law of the reset times independent of the
walk, whenever absorption occurs almost surely. To our knowledge, none
of the works cited above states this protocol-independent result, or
reports a site at which the splitting probability of a biased walk is
independent of the resetting rate. Under geometric resetting, the same
duality also determines the unique sign change of the response away
from the midpoint.

The channel duality identified here has been the starting point of
subsequent work. The invariance extends to multi-site resetting, for a
family of reset-neutral distributions~\cite{PaperII}; the geometry of
the admissible reset distributions has been
characterized~\cite{PaperIII}; and the duality has been recast
successively as an intertwining relation between
resolvents~\cite{PaperIV}, as a Fisher--Rao geometry of the neutral
set~\cite{PaperV}, and as an algebraic structure in its own
right~\cite{PaperVI}.

Geometric resetting --- where, at each discrete time step, the walker is
reset to its initial position with probability $\gamma$ --- is the
natural discrete-time analogue of Poissonian resetting, and it is the
protocol for which we work out the ruin probability in closed form. The
central result of this paper, however, holds for any resetting
mechanism that returns the walker to its starting point independently
of its motion.

\subsection{Model and approach}
\label{sec:model-approach}

Absorption probabilities record not when a stochastic process
terminates but where. In an unbounded domain that distinction is often
vacuous; in a finite one it is the whole question, since the two
boundaries compete and the outcome is decided by the interplay between
the bias, the confinement, and the starting position. Resetting
interferes with all three at once, by repeatedly returning the walker
to where it began. It is therefore natural to ask whether that
interference has a preferred direction everywhere, or whether some
starting position is immune to it. The gambler's ruin problem with
resetting is the simplest setting in which the question can be posed
and answered exactly: the reset-free problem is classical, and the
geometry is a single interval with one reflection.

We consider a biased nearest-neighbor random walk on $\{0,1,\dots,a\}$
with absorbing boundaries at $0$ and $a$, which steps to the right with
probability $p$ and to the left with probability $q=1-p$, and which is
returned to its starting point $z$ at a set of reset times independent
of its motion: the resetting mechanism, deterministic or random, is
otherwise arbitrary. Two mechanisms are singled out: geometric
resetting, in which each step is a reset with probability
$\gamma\in(0,1)$, and, in continuous time, Poissonian resetting of a
diffusion.

Our analysis proceeds in three stages.

\paragraph{Reset-free channels and their spectral representation.}
Every resetting mechanism acts on the walk only through the times at
which it interrupts it; between resets the walk is the classical
reset-free walk. We therefore begin with the reset-free absorption
laws, separated by boundary. We construct a weighted Hilbert space in
which the transition operator becomes self-adjoint, and diagonalization
expresses the probability of absorption at $0$ and at $a$ at each step
as sums over spectral modes weighted by amplitudes $A_\nu(z)$ and
$B_\nu(z)$ --- the two \emph{absorption channels}. For geometric
resetting a renewal identity then gives the ruin probability in closed
form,
\[
q_z(\gamma) =
\frac{\sum_{\nu=1}^{a-1} A_\nu(z)\, g_\nu(\gamma)}
{\sum_{\nu=1}^{a-1} [A_\nu(z)+B_\nu(z)]\, g_\nu(\gamma)},
\]
where $g_\nu(\gamma)$ encodes the resetting dependence through the
eigenvalues of the underlying walk.

\paragraph{Channel duality and the invariance.}
Rather than treat absorption as a single lumped event, we keep the two
channels separate. The two channel amplitudes satisfy an exact duality
relation for every starting site, $B_\nu(z) = (p/q)^{a-z}A_\nu(a-z)$
(derived in Section~\ref{sec:spectral}), which holds equally step by
step in time for the absorption probabilities themselves. At the point
$z=a/2$ fixed by the reflection $z\mapsto a-z$ the two channels differ
only by the constant factor $c=(q/p)^{a/2}$, at every mode and at every
time. It is this relation between channels --- not a cancellation
particular to either one --- that fixes the ruin probability from the
midpoint: resetting acts only on the timing of a trajectory, and from
$a/2$ the timing carries no information about the side of absorption.
When $a$ is even, $z=a/2$ is an occupied lattice site, and $q_{a/2}$
coincides with the classical, reset-free ruin probability for every
$p\in(0,1)$ and for any resetting mechanism independent of the walk
under which absorption occurs almost surely, geometric resetting at any
rate $\gamma\in(0,1)$ being one instance.
When $a$ is odd, the fixed point of the reflection is still $z=a/2$,
but it is now a half-integer, and no lattice site realizes it; under
geometric resetting the sign change of the response is then observed
between the two central sites (Corollary~\ref{cor:unique}). The parity
of $a$ therefore decides whether the neutral point is occupied, not
whether it exists.

This structural origin also explains why the invariance is not tied to
the particular nearest-neighbor walk studied here. Once a duality
relation between the two absorption channels is available, the
cancellation at the reflection-fixed point is algebraic and does not
itself rely on the Markov property or on the detailed form of the
transition law. The biased walk studied here is one concrete realization of
this structure. The same mechanism survives the diffusive limit
(Section~\ref{sec:continuum}) and extends to
resetting distributed over several sites, where it singles out a
family of reset-neutral distributions
(Remark~\ref{rem:multisite-scope} and Ref.~\cite{PaperII}).

\subsection{Results}

The contribution here is a channel-resolved spectral representation for
a biased nearest-neighbor walk with two absorbing boundaries, together
with the twisted reflection duality that it reveals: at the midpoint it
yields a ruin probability that no resetting mechanism can modify, and
away from it the sign of the response.

Our results, in the order in which they are established below, are
the following. We identify the structure behind them: a duality
between the two absorption channels, holding mode by mode in the
spectral representation and step by step in time. At the
reflection-fixed point this duality alone fixes the ruin probability:
for even $a$, provided absorption occurs almost surely, the ruin
probability from $z=a/2$ is independent of the resetting mechanism and
equals its classical, reset-free value, for every bias and every
sequence of reset times independent of the walk;
conditionally on the reset times, the side of absorption is even
independent of when absorption occurs. Away from the midpoint, any
resetting to the starting point moves the ruin probability towards the
nearer boundary, so that the effect of resetting changes sign exactly
at $a/2$ whatever the mechanism. The strong property is the
proportionality of the two channels at every time; the invariance of
the splitting probability, for every resetting mechanism, is its
consequence. Resetting cannot modify the ruin probability from the
midpoint: whatever the law of the reset times, it is exactly the
classical ruin probability of the reset-free walk. For geometric
resetting this is
the invariance $\partial_\gamma q_{a/2}\equiv0$, and the duality at
every site gives more: a temporal criterion for the sign of the
resetting response, together with a proof --- not merely numerical
evidence --- of the uniqueness of its sign change, and the parity
effect for odd domains. We show that the invariance survives the
diffusive limit for any resetting mechanism, and its uniqueness under
Poissonian resetting, while the parity effect does not. Finally, we
develop the spectral method itself, built on a weighted Hilbert space
through the Doob transform, as a tool applicable beyond this
particular model.

Throughout, a single dichotomy organizes the results. The location of
the separatrix, $z^\dagger=a/2$, is fixed by geometry alone: it does
not depend on the bias, on the rate or the law of resetting, or, in the
continuum, on the drift and the volatility. The magnitude of the
response on either side of it is fixed by the dynamics.

\subsection{Organization}

Section~\ref{sec:model} defines the walk with an arbitrary resetting
mechanism and the reset-free absorption laws through which every
mechanism acts. Section~\ref{sec:spectral} constructs the spectral
decomposition on a weighted Hilbert space, separated by absorption
channel, and establishes the channel duality.
Section~\ref{sec:fixed-point} shows that at the reflection-fixed point
the duality makes the ruin probability independent of the resetting
mechanism, and that away from it every mechanism favors the nearer
boundary. Section~\ref{sec:invariance} treats geometric resetting: the
renewal identity and the closed form of the ruin probability, and the
structure of the response --- the sign change, the temporal criterion,
and the parity effect for odd domains.
Section~\ref{sec:continuum} shows that in the diffusive limit the
invariance holds for every resetting mechanism and, under Poissonian
resetting, its uniqueness survives, while the parity effect does
not. Section~\ref{sec:numerics} validates the closed-form solution
numerically, and Section~\ref{sec:conclusions} discusses implications,
limitations and open questions. Proofs of the discrete results are
collected in Appendix~\ref{app:complete_proof}, an illustration of
multi-site resetting in Appendix~\ref{app:multisite}, and the
numerical tables in Appendix~\ref{app:tables}.

\section{The model}
\label{sec:model}

\subsection{The walk and the resetting mechanism}
\label{sec:model-classical}

The ruin problem describes a stochastic system that terminates upon
reaching one of two absorbing states. A particle performs a random walk
on $\{0,1,\dots,a\}$ with transition probabilities $p$ (right) and
$q=1-p$ (left), and the process ends when it reaches either $0$ (ruin)
or $a$ (success). In the gambling interpretation, a gambler with initial
capital $z$ plays against an adversary holding $a-z$, so that the total
wealth $a$ is fixed, and the game ends when one player is ruined. The
state space is $E=\{0,1,\dots,a\}$, with $0$ and $a$ absorbing and
$\{1,\dots,a-1\}$ transient.

We superimpose on this walk a resetting mechanism that returns the
walker to its starting point at times that do not depend on its motion.

\begin{definition}[Random walk with resetting]
\label{def:process}
Let $z\in\{1,\dots,a-1\}$, let $(\xi_n)_{n\ge1}$ be i.i.d.\ with
$\mathbb P(\xi_n=1)=p$ and $\mathbb P(\xi_n=-1)=q$, and let
$\mathcal R\subset\{1,2,\dots\}$ be a set of reset times, finite or
infinite, deterministic or random, independent of $(\xi_n)$. The walk
with resetting started at $z$ is defined by $X_0=z$ and, for $n\ge1$
until absorption,
\[
X_n=\begin{cases} z, & n\in\mathcal R,\\ X_{n-1}+\xi_n, & n\notin\mathcal R;\end{cases}
\]
it is absorbed at the first time $\tau$ at which $X_\tau\in\{0,a\}$.
The set $\mathcal R$ is the \emph{resetting mechanism}; the empty set
gives the classical walk without resetting.
\end{definition}

The ruin probability of the walk with resetting is
$q_z(\mathcal R):=\mathbb P_z(X_\tau=0)$. We write $q_z$ when the
mechanism is clear from the context, and $q^{(0)}_z$ for the classical
ruin probability without resetting~\cite{Feller1968}. Two mechanisms
are singled out. \emph{Geometric resetting}, in which each step is a
reset with probability $\gamma$ independently of everything else, is
the discrete-time model studied in detail in Section~\ref{sec:invariance},
where its ruin probability is written $q_z(\gamma)$. Its continuous-time
counterpart, Poissonian resetting of a diffusion, is studied in
Section~\ref{sec:continuum}. Figure~\ref{fig:sample-path} illustrates a
sample path.

\begin{figure}[!tb]
\centering
\scalebox{0.8}{
\begin{tikzpicture}[x=0.9cm, y=0.6cm, >=stealth]
  \draw[->, thick] (0,0) -- (9.2,0) node[below] {$n$};
  \draw[->, thick] (0,0) -- (0,6.0) node[left] {$X_n$};
  \foreach \x in {1,...,8} \draw (\x,0.08) -- (\x,-0.08) node[below] {\scriptsize \x};
  \foreach \y/\l in {0/0,2/z,5/a} \draw (0.08,\y) -- (-0.08,\y) node[left] {\small $\l$};
  \draw[gray, thin] (0,5) -- (8.6,5);
  \draw[gray, dotted, thin] (0,2) -- (8.6,2);
  \draw[thick, black] (0,2) -- (1,3) -- (2,4) -- (3,3);
  \draw[thick, red, dashed, ->] (3,3) -- (4,2);
  \node[red, above right, font=\small] at (3.5,2.6) {reset, $4\in\mathcal R$};
  \draw[thick, black] (4,2) -- (5,1) -- (6,2) -- (7,1) -- (8,0);
  \foreach \x/\y in {0/2,1/3,2/4,3/3,4/2,5/1,6/2,7/1} \filldraw (\x,\y) circle (1.4pt);
  \filldraw[blue] (8,0) circle (2pt) node[above right, font=\small] {$\tau=8$};
  \node[font=\small, above] at (1.6,4.1) {$X_n=X_{n-1}+\xi_n$};
\end{tikzpicture}}
\caption{Sample path of the walk of Definition~\ref{def:process} with
$a=5$, $z=2$ and a reset at time $4$. At the movement steps
$n=1,2,3$ the walker evolves as $X_n=X_{n-1}+\xi_n$; the step $n=4$ is
a reset step, at which $X_4=z$ (red dashed arrow); the walk then
resumes and is absorbed at $0$ at time $\tau=8$.}
\label{fig:sample-path}
\end{figure}

\begin{remark}[Resetting and absorption cannot coincide]
By Definition~\ref{def:process}, each time step $n$ is either a reset
step, in which case $X_n=z$, or a movement step, in which case
$X_n=X_{n-1}+\xi_n$ --- these two alternatives are complementary, so
they cannot occur together at the same $n$. Moreover
$z\in\{1,\dots,a-1\}$ is an interior point, so a reset can never itself
be an absorption event. The only event that matters is therefore
whether the walker is absorbed before the next reset; no further
convention on simultaneity is needed.
\end{remark}

\subsection{Reset-free absorption laws}
\label{sec:model-generating}

Between two resets the walk evolves as the classical reset-free walk,
and every resetting mechanism acts on it only through the reset times.
The basic objects are therefore the absorption laws of the reset-free
walk. Let $T_0$, $T_a$ and $T:=\min\{T_0,T_a\}$ be the first-passage
times of the walk \emph{without} resetting, started at $z$, and write
\[
\begin{aligned}
u_{z,k} &:= \mathbb{P}_z(T_0=k,\ T_0<T_a),\\
v_{z,k} &:= \mathbb{P}_z(T_a=k,\ T_a<T_0),\\
s_{z,k} &:= \mathbb{P}_z(T=k)=u_{z,k}+v_{z,k},\qquad k\ge1,
\end{aligned}
\]
for the probabilities of absorption at $0$, at $a$, and at either
boundary at step $k$, with generating functions
\begin{equation*}
\begin{aligned}
U_z(s) &:= \sum_{k=1}^{\infty} u_{z,k}\, s^k,\qquad
V_z(s) := \sum_{k=1}^{\infty} v_{z,k}\, s^k,\\
S_z(s) &:= U_z(s)+V_z(s),
\end{aligned}
\end{equation*}
for $s\in(0,1)$. In a cycle of $L$ steps between two resets, the walk
is absorbed at $0$ with probability $\sum_{k\le L}u_{z,k}$ and at $a$
with probability $\sum_{k\le L}v_{z,k}$; this is how the reset-free
laws enter every result below. Their construction for the
nearest-neighbor walk uses its Markov structure; the way they are
combined across cycles does not.

The coefficients $u_{z,k}$ and $v_{z,k}$ admit explicit spectral
decompositions once the transition operator of the underlying walk is
diagonalized. That operator is not self-adjoint under the standard
inner product for a biased walk; in Section~\ref{sec:spectral} we
symmetrize it through a diagonal similarity transform, diagonalize it
on a weighted Hilbert space, and obtain the spectral expressions of the
two absorption channels, together with the duality that relates them.

\section{Channel-resolved spectral decomposition and the channel duality}
\label{sec:spectral}

Section~\ref{sec:model} reduced every resetting mechanism to the
absorption laws $u_{z,k}$ and $v_{z,k}$ of the walk without resetting
(Section~\ref{sec:model-generating}). We now evaluate these explicitly,
keeping the two boundaries separate throughout. Diagonalizing the transition
operator gives each of $u_{z,k}$, $v_{z,k}$ as a sum over spectral
modes weighted by a channel-specific amplitude, $A_\nu(z)$ for channel
$0$ and $B_\nu(z)$ for channel $a$; the relation between these two
families of amplitudes, not their individual values, is what produces
the results of Sections~\ref{sec:fixed-point} and~\ref{sec:invariance}.

\subsection{Symmetrization and the Doob transform}

For a biased random walk with $p\neq q$, the transition operator
restricted to the transient states $\{1,\dots,a-1\}$,
\[
Q(x,y) = \begin{cases} p, & y=x+1,\\ q, & y=x-1,\\ 0, & \text{otherwise,}\end{cases}
\qquad x,y\in\{1,\dots,a-1\},
\]
is not self-adjoint under the standard inner product on
$\ell^2(\{1,\dots,a-1\})$. The natural setting for its spectral analysis
is instead a \emph{weighted} Hilbert space~\cite{KarlinMcGregor1959},
in which $Q$ becomes self-adjoint after conjugation by a suitable
diagonal weight: as a nearest-neighbor chain, $Q$ is reversible with
respect to $\pi(x)=(p/q)^x$, and conjugation by $\pi^{1/2}$
symmetrizes it. We use the positive diagonal similarity (gauge)
transform $D^{-1}QD$ of Proposition~\ref{prop:doob_symmetry} below. It
is closely related to the Doob
$h$-transform~\cite{Doob1953, ChetriteTouchette2015}; here we use
``Doob transform'' in this spectral-conjugation sense, without
interpreting $\widetilde Q$ as a stochastic transition matrix. Spectral methods have in
recent years been brought to bear on random walks under resetting,
notably on arbitrary networks~\cite{Riascos2020, Riascos2022}, where the
stationary distribution and the first-passage statistics are obtained
from the spectrum of the reset-free walk; closed-form propagators for
confined biased lattice walks, without resetting, are given
in~\cite{SarvaharmanGiuggioli2020}.

The relevant weight here is
\begin{equation}
\label{eq:doob_h}
h(x) = \left(\frac{q}{p}\right)^{x/2}, \qquad x\in\{1,\dots,a-1\},
\end{equation}
and the conjugated operator is
\begin{equation}
\label{eq:doob_transform}
\widetilde{Q}(x,y) := \frac{h(y)}{h(x)}\,Q(x,y).
\end{equation}

\begin{proposition}[Symmetry via the Doob transform]
\label{prop:doob_symmetry}
The operator $\widetilde{Q}$ is self-adjoint on $\ell^2(\{1,\dots,a-1\})$
with the standard inner product. Equivalently, $\widetilde{Q}=D^{-1}QD$
with
\[
D = \mathrm{diag}\!\left(\left(\tfrac{q}{p}\right)^{1/2},
\left(\tfrac{q}{p}\right)^{1}, \dots,
\left(\tfrac{q}{p}\right)^{(a-1)/2}\right),
\]
$Q$ the tridiagonal matrix above, and, explicitly,
\[
\widetilde{Q}(x,x+1) = \widetilde{Q}(x+1,x) = \sqrt{pq}, \qquad x=1,\dots,a-2,
\]
all other entries zero.
\end{proposition}
\begin{proof}
Appendix~\ref{app:spectral-proofs}.
\end{proof}

\begin{remark}[Connection to the Girsanov transform]
Under the diffusive scaling of Section~\ref{sec:continuum}, the biased
walk converges to $dX_t=\mu\,dt+\sigma\,dW_t$, with the bias and the
drift related by $\log(p/q)\leftrightarrow2\mu/\sigma^2$
(Remark~\ref{rem:scaling}), and the discrete diagonal conjugation by
$h(x)=(q/p)^{x/2}$ converges to the exponential gauge transformation
$w(x)=e^{-\mu x/\sigma^2}$, which removes the first-order drift term
from the backward generator. This conjugation is related to the change
of measure underlying Girsanov's theorem, but is used here as a
symmetrization of the operator; it is the same weight that symmetrizes
the continuous-time problem of Section~\ref{sec:continuum}.
\end{remark}

\subsection{Spectrum}

\begin{proposition}[Eigenvalues and eigenfunctions of $\widetilde{Q}$]
\label{prop:eigenvalues}
$\widetilde{Q}$ has eigenvalues and (unnormalized) eigenfunctions
\begin{equation}
\label{eq:eigenvalues}
\lambda_\nu = 2\sqrt{pq}\,\cos\!\left(\frac{\pi\nu}{a}\right),\end{equation}
\begin{equation*}\psi_\nu(x) = \sin\!\left(\frac{\pi\nu x}{a}\right),
\qquad \nu,x=1,\dots,a-1.
\end{equation*}
\end{proposition}
\begin{proof}
Appendix~\ref{app:spectral-proofs}; the eigenfunctions are the discrete sine
modes classical for tridiagonal symmetric operators of this
form~\cite{Feller1968, KarlinMcGregor1959}, and the eigenvalues follow
from the transform~\eqref{eq:doob_transform} applied to the
nearest-neighbor structure.
\end{proof}

\begin{remark}[Spectral radius]
\label{rem:spectral-radius}
$|\lambda_\nu|\le2\sqrt{pq}<1$ for all $\nu$ whenever $p\neq1/2$, which
guarantees convergence of the generating-function series below for
every $s\in(0,1)$. At $p=1/2$ the bound $2\sqrt{pq}<1$ degenerates to
equality, but the eigenvalues themselves do not: $\lambda_\nu=\cos(\pi
\nu/a)$, and $|\lambda_\nu|<1$ for every $\nu=1,\dots,a-1$, since
$\nu/a\in(0,1)$. Convergence therefore holds uniformly across the whole
range of $p$, including the symmetric case.
\end{remark}

\subsection{Channel amplitudes and the duality relation}
\label{sec:channel-duality}

Absorption at $0$ in exactly $k$ steps requires $X_{k-1}=1$ followed by
a leftward step, so $u_{z,k}=q\,\mathbb{P}_z(X_{k-1}=1,\,T>k-1)$;
expanding the initial condition in the eigenbasis of $\widetilde{Q}$
and undoing the conjugation via $h$ gives the spectral representation of
both absorption channels.

\begin{theorem}[Channel amplitudes]
\label{thm:spectral_absorption}
For $k\ge1$ and $z\in\{1,\dots,a-1\}$,
\begin{equation}
\label{eq:u_spectral}
u_{z,k} = \sqrt{pq}\sum_{\nu=1}^{a-1} A_\nu(z)\,\lambda_\nu^{\,k-1},\end{equation}
\begin{equation*}v_{z,k} = \sqrt{pq}\sum_{\nu=1}^{a-1} B_\nu(z)\,\lambda_\nu^{\,k-1},
\end{equation*}
with
\begin{equation}
\label{eq:A_nu}
A_\nu(z) = \frac{2}{a}\left(\frac{q}{p}\right)^{z/2}
\sin\!\left(\frac{\pi\nu z}{a}\right)\sin\!\left(\frac{\pi\nu}{a}\right),\end{equation}
\begin{equation*}B_\nu(z) = \frac{2}{a}\left(\frac{p}{q}\right)^{(a-z)/2}
\sin\!\left(\frac{\pi\nu(a-z)}{a}\right)\sin\!\left(\frac{\pi\nu}{a}\right).
\end{equation*}
\end{theorem}
\begin{proof}
Appendix~\ref{app:spectral-proofs}.
\end{proof}

$A_\nu(z)$ and $B_\nu(z)$ are the \emph{channel amplitudes}: the weight
each spectral mode $\nu$ carries toward absorption at $0$ and at $a$
respectively, for a walk started at $z$. The two families are related
by an exact identity.
Replacing $z$ by $a-z$ in $B_\nu$,
\[
B_\nu(a-z) = \frac{2}{a}\left(\frac{p}{q}\right)^{z/2}
\sin\!\left(\frac{\pi\nu z}{a}\right)\sin\!\left(\frac{\pi\nu}{a}\right),
\]
which has exactly the same sine factors as $A_\nu(z)$ in
\eqref{eq:A_nu} and differs from it only in the sign of the exponent of
$q/p$. Dividing the two expressions gives the relation between the
channels.

\begin{theorem}[Channel duality]
\label{thm:channel-duality}
For every $z\in\{1,\dots,a-1\}$ the two absorption channels are
exchanged by the reflection $z\mapsto a-z$, up to the weight
$(p/q)^{a-z}$, in the following equivalent forms.
\begin{enumerate}[label=(\roman*)]
\item Mode by mode: for every $\nu=1,\dots,a-1$,
\begin{equation}
\label{eq:channel-duality}
B_\nu(z) = \left(\frac{p}{q}\right)^{a-z} A_\nu(a-z).
\end{equation}
\item Step by step in time: for every $k\ge1$,
\begin{equation}
\label{eq:duality-time}
v_{z,k} = \left(\frac{p}{q}\right)^{a-z} u_{a-z,k}.
\end{equation}
\item For the generating functions:
$V_z(s)=(p/q)^{a-z}\,U_{a-z}(s)$ for every $s\in(0,1)$.
\end{enumerate}
\end{theorem}
\begin{proof}
(i) is the computation above. (ii) follows by inserting (i) into
Theorem~\ref{thm:spectral_absorption}. It can also be seen directly: a
path from $z$ absorbed at $a$ at step $k$ makes $a-z$ more steps to the
right than to the left; its reflection $x\mapsto a-x$ is a path from
$a-z$, of the same length, absorbed at $0$ at step $k$, with the
numbers of right and left steps exchanged, and the reflection is a
bijection between the two families of paths, so the two probabilities
differ by the factor $(p/q)^{a-z}$. (iii) follows by summing (ii)
against $s^k$.
\end{proof}

Equation~\eqref{eq:channel-duality} says that the amplitude for
absorption at $a$ from $z$ equals the amplitude for absorption at $0$
from the mirrored point $a-z$, up to the mode-independent factor
$(p/q)^{a-z}$: exchanging the two boundaries is, at the level of
channel amplitudes, the same operation as reflecting the starting
point. This is why it holds mode by mode, for every $\nu$ separately,
rather than only after summing the spectral series: both $A_\nu$ and
$B_\nu$ factor as (bias weight)$\times$(the same pair of sine terms up
to the relabeling $z\leftrightarrow a-z$), because the Doob weight
\eqref{eq:doob_h} and the reflection symmetry of the sine eigenfunctions (Proposition~\ref{prop:eigenvalues}) act independently of one another. At the
reflection-fixed point $z=a/2$ the theorem gives
$A_\nu(a/2)=(q/p)^{a/2}B_\nu(a/2)$ for every $\nu$ and
$u_{a/2,k}=(q/p)^{a/2}v_{a/2,k}$ for every $k$: there the two channels
differ only by a constant factor, at every mode and at every time.
Section~\ref{sec:fixed-point} shows that this alone fixes the ruin
probability from $a/2$, whatever the resetting mechanism;
Section~\ref{sec:invariance} uses the duality at every site to
determine the structure of the response under geometric resetting.

\subsection{Generating functions}

Substituting \eqref{eq:u_spectral} into $U_z(s)=\sum_k u_{z,k}s^k$ and
summing the resulting geometric series in $k$ (justified for
$s\in(0,1)$ by Remark~\ref{rem:spectral-radius}),
\[
\begin{aligned}
U_z(s) = \sqrt{pq}\,s\sum_{\nu=1}^{a-1} \frac{A_\nu(z)}{1-\lambda_\nu s},\\[2pt]
S_z(s) = \sqrt{pq}\,s\sum_{\nu=1}^{a-1} \frac{A_\nu(z)+B_\nu(z)}{1-\lambda_\nu s},
\end{aligned}
\]
using $s_{z,k}=u_{z,k}+v_{z,k}$. These are the only properties of the
reset-free walk used in the rest of the paper.

The closed form of the ruin probability, for geometric resetting, is
obtained from these sums in Section~\ref{sec:invariance}; the channel
duality explains its structure. The time-domain form of the duality
can also be read off from the paths (Theorem~\ref{thm:channel-duality}),
but the spectral decomposition is where it appears mode by mode, and in
that form it extends beyond the present walk~\cite{PaperII}. The
channel duality is what Section~\ref{sec:fixed-point} evaluates at
$z=a/2$ to obtain the midpoint invariance for every resetting
mechanism, and what Section~\ref{sec:continuum} recovers, mode by mode
replaced by the continuous spectral parameter $\beta$, in the diffusive
limit.

\section{Invariance at the reflection-fixed point}
\label{sec:fixed-point}

At $z=a/2$ the channel duality of Theorem~\ref{thm:channel-duality}
takes its strongest form. By~\eqref{eq:duality-time} with $z=a/2$,
\begin{equation}
\label{eq:midpoint-duality}
u_{a/2,k}=\Bigl(\frac qp\Bigr)^{a/2}v_{a/2,k}\qquad\text{for every }k\ge1:
\end{equation}
started from the midpoint, the reset-free walk is absorbed at $0$ and
at $a$ in the same proportion at every step. This section shows that
this alone fixes the ruin probability from $a/2$, whatever the
resetting mechanism. Throughout, $a$ is even, so that $a/2$ is a
lattice site at which the walk can start; the odd case is discussed at
the end of the section.

Resetting mechanisms are those of Definition~\ref{def:process}; the
absence of resetting is the empty mechanism, and geometric resetting
(Definition~\ref{def:geometric}) is the mechanism in which the times
between resets are i.i.d.\ and geometric.

\begin{theorem}[Independence of the resetting mechanism]
\label{thm:protocol}
Let $a$ be even, let the walk start at $a/2$ and be reset to $a/2$
according to an arbitrary resetting mechanism, and write $\tau$ for the
absorption time, $N$ for the number of resets before absorption, and
\[
C^\ast:=\frac{(q/p)^{a/2}}{1+(q/p)^{a/2}} .
\]
\begin{enumerate}[label=(\roman*)]
\item For every $n\ge1$ and $m\ge0$, conditionally on the reset times
and hence also unconditionally,
$\mathbb P(\tau=n,\,N=m,\,X_\tau=0)
=(q/p)^{a/2}\,\mathbb P(\tau=n,\,N=m,\,X_\tau=a)$. In particular, the ratio of the two absorption probabilities,
$\mathbb P(X_\tau=0)/\mathbb P(X_\tau=a)=(q/p)^{a/2}$, is the same for
every resetting mechanism.
\item The side of absorption is independent of the absorption time
and of the number of resets:
$\mathbb P(X_\tau=0\mid\tau=n)=\mathbb P(X_\tau=0\mid N=m)=C^\ast$
whenever the conditioning event has positive probability.
\item If absorption occurs almost surely, the ruin probability from
$a/2$ is $C^\ast$, for every resetting mechanism.
\end{enumerate}
\end{theorem}

\begin{proof}
Fix a realization of the reset times; since they are independent of
the steps, the steps remain i.i.d.\ under this conditioning. If $n$ is
a reset time, absorption at time $n$ is impossible and both sides of
(i) vanish. Otherwise let $t_m$ be the last reset time before $n$
(with $t_0=0$), so that absorption at time $n$ can only occur in the
cycle that starts at $t_m$, after $m$ resets, at the local step
$k=n-t_m$. Each cycle starts afresh at $a/2$, so the event
$\{\tau=n,\,N=m,\,X_\tau=0\}$ has conditional probability
$\prod_{j=1}^{m}\sigma_j\cdot u_{a/2,k}$, where $\sigma_j$ is the
probability that the $j$-th cycle ends in a reset rather than in
absorption; replacing $u_{a/2,k}$ by $v_{a/2,k}$ gives the probability
of $\{\tau=n,\,N=m,\,X_\tau=a\}$. By~\eqref{eq:midpoint-duality} the two
differ by the factor $(q/p)^{a/2}$, for every realization, and
averaging over the reset times preserves the identity. This is (i).
Summing (i) over $m$, respectively over $n$, and dividing gives (ii).
Summing over both gives
$\mathbb P(X_\tau=0)=(q/p)^{a/2}\,\mathbb P(X_\tau=a)$, and if the two
probabilities add up to one this is (iii).
\end{proof}

\begin{corollary}[The classical ruin probability]
\label{cor:classical}
$C^\ast=q^{(0)}_{a/2}$, the ruin probability of the walk without
resetting started at $a/2$~\cite{Feller1968}. No resetting mechanism can
therefore move the ruin probability from the midpoint away from its
classical value.
\end{corollary}

\begin{proof}
The absence of resetting is a resetting mechanism, under which
absorption occurs almost surely on a finite interval; apply
Theorem~\ref{thm:protocol}(iii).
\end{proof}

\begin{corollary}[Midpoint invariance under geometric resetting]
\label{cor:geometric-midpoint}
Let $a$ be even. For every $p\in(0,1)$ and every $\gamma\in(0,1)$,
\[
q_{a/2}(\gamma)=C^\ast=\frac{(q/p)^{a/2}}{1+(q/p)^{a/2}},
\]
and hence $\partial_\gamma q_{a/2}(\gamma)\equiv0$.
\end{corollary}

\begin{proof}
Geometric resetting is a resetting mechanism, and absorption occurs
almost surely under it (Section~\ref{sec:model-renewal}); apply
Theorem~\ref{thm:protocol}(iii). The closed form gives an independent
proof, Corollary~\ref{cor:midpoint} in Appendix~\ref{app:complete_proof}.
\end{proof}

Besides geometric resetting, Theorem~\ref{thm:protocol} covers any
renewal law of the times between resets, periodic resetting, and
resetting at a rate that varies in time~\cite{Pal2016}; its continuous
counterpart is Proposition~\ref{prop:cont-protocol}. The hypothesis of
almost-sure absorption in part (iii) cannot be dropped: if resets are
so frequent that the walker can never complete the $a/2$ consecutive
steps needed to reach a boundary, absorption never occurs, and only the
conditional statement (ii) remains.

The invariance under geometric resetting and the invariance under every
resetting mechanism are in fact the same property, and both are
equivalent to the proportionality of the two channels at every step.

\begin{corollary}[Characterization of reset-neutral starting points]
\label{cor:equivalence}
For $z\in\{1,\dots,a-1\}$ the following are equivalent:
\begin{enumerate}[label=(\alph*)]
\item under geometric resetting, $q_z(\gamma)$ does not depend on
$\gamma\in(0,1)$;
\item $u_{z,k}=c\,v_{z,k}$ for every $k\ge1$, with $c$ independent of
$k$; that is, for the reset-free walk started at $z$, the side of
absorption is independent of the absorption time;
\item $q_z$ is the same for every resetting mechanism under which
absorption occurs almost surely.
\end{enumerate}
For the biased walk they hold exactly when $a$ is even and $z=a/2$.
\end{corollary}

\begin{proof}
If (a) holds with constant value $C$, then $U_z(s)=C\,S_z(s)$ for
$s\in(0,1)$, and by uniqueness of power series $u_{z,k}=C\,s_{z,k}$ for
every $k$, which is (b) with $c=C/(1-C)$. If (b) holds, the proof of
Theorem~\ref{thm:protocol}, which uses~\eqref{eq:midpoint-duality} and
nothing else about the starting point, applies verbatim at $z$ and
gives (c). Geometric resetting at any rate is a resetting mechanism, so
(c) implies (a). Finally, by Theorem~\ref{thm:main_result}(i) the
response $\partial_\gamma q_z$ vanishes identically only at $z=a/2$ for
even $a$.
\end{proof}

The argument uses no property of the walk between resets other than
that each reset starts it afresh, independently of its past: the
equivalence of (a), (b) and (c) holds for any process with two
absorbing exits that regenerates at each reset, Markovian or not.
Invariance with respect to the rate of a single resetting protocol thus
forces invariance with respect to every protocol.

\begin{remark}[Interpretation]
\label{rem:protocol-interpretation}
Resetting acts only on the timing of a trajectory. By part (ii), at the
reflection-fixed point the timing carries no information about the
side of absorption, and resetting therefore cannot affect it. For the
reset-free walk alone, the independence of exit time and exit side
from $a/2$ is elementary; what the theorem uses is its consequence for
resetting, and it is the channel duality that supplies it. What the
theorem does not give is what happens away from the midpoint. The
direction of the effect relative to no resetting is given, for every
resetting mechanism, by Theorem~\ref{thm:direction} below. The response
to a change in the intensity of resetting, the uniqueness of its change
of sign, the odd case and the size of the effect depend on how the
mechanism is parametrized, and on the duality at every site rather
than only at $a/2$; for geometric resetting they are the content of
Section~\ref{sec:invariance}, and for Poissonian resetting of
Section~\ref{sec:continuum}.
\end{remark}

\begin{remark}[Parity]
\label{rem:protocol-parity}
The theorem requires $a$ to be even, so that $a/2$ is a site at which
the walk can start. For odd $a$ no lattice site realizes $a/2$, and the
statement has no direct counterpart for the discrete chain. Under
geometric resetting the invariance persists at $a/2$ in the sense of
the analytic continuation of $q_z(\gamma)$
(Section~\ref{sec:invariance-parity}), and in the continuum, where every
point is available, it holds without restriction
(Proposition~\ref{prop:cont-protocol}).
\end{remark}

\begin{remark}[A martingale reading]
\label{rem:martingale}
Let $\mathcal F_n$ be generated by the reset times and by
$X_0,\dots,X_n$, and let $Y_n=\mathbb P(X_\tau=0\mid\mathcal F_n)$, the
Doob martingale of the side of absorption. When absorption occurs
almost surely, Theorem~\ref{thm:protocol}(i), applied to the
remaining reset times, shows that $Y$ returns to the value $C^\ast$ at
every reset: each reset pins the martingale back to its initial value,
and this is why neither the number of resets nor their timing can move
it. The same structure appears in the classical derivation of the
ruin probability and of the duration from the exponential martingales
$t^{\,n}r_\pm^{\,X_n}$ of the reset-free walk, whose two roots satisfy
$r_+r_-=q/p$ (Lemma~\ref{lem:roots}); this symmetry between the two
roots is the channel duality in martingale form. Together with
Corollary~\ref{cor:equivalence}, which characterizes reset-neutral
starting points as those from which the side of absorption is
independent of the absorption time, this suggests a probabilistic
characterization of the spectral duality structures that are realized
by actual processes. We do not pursue it here.
\end{remark}

Away from the midpoint, the same two ingredients --- the channel
duality and the regeneration at each reset --- determine the direction
in which any resetting mechanism moves the ruin probability. For a
walk started at $z$, let $\mathcal U_z(L)=\sum_{k\le L}u_{z,k}$ and $\mathcal V_z(L)=\sum_{k\le L}v_{z,k}$ be the probabilities that the reset-free
walk is absorbed at $0$, respectively at $a$, within its first $L$
steps, and note that $q^{(0)}_z=\mathcal U_z(\infty)$ is the classical ruin probability.

\begin{theorem}[Direction of the effect for any resetting mechanism]
\label{thm:direction}
Let $z\in\{1,\dots,a-1\}$, and let the walk start at $z$ and be reset
to $z$ according to an arbitrary resetting mechanism under which
absorption occurs almost surely. Then its ruin probability $q_z$
satisfies
\[
q_z\ \ge\ q^{(0)}_z\quad\text{if } z<a/2,
\qquad
q_z\ \le\ q^{(0)}_z\quad\text{if } z>a/2,
\]
with equality at $z=a/2$ when $a$ is even (Theorem~\ref{thm:protocol}).
Whatever the resetting mechanism, resetting to the starting point
favors the nearer boundary, for every bias, and its effect changes
sign exactly at $a/2$.
\end{theorem}

\begin{proof}
Let $z<a/2$ and $w=a-z>z$; the case $z>a/2$ follows by the same
argument with the two boundaries exchanged. A walk from $w$ must visit
$z$ before reaching $0$. Writing $f_j$ for the probability that the
reset-free walk from $w$ first hits $z$ at step $j$ without having
reached $a$, the strong Markov property gives
$\mathcal U_w(L)=\sum_j f_j\,\mathcal U_z(L-j)\le\bigl(\sum_j f_j\bigr)\mathcal U_z(L)$ and $\mathcal U_w(\infty)=\bigl(\sum_j f_j\bigr)\mathcal U_z(\infty)$, hence $\mathcal U_w(L)/\mathcal U_w(\infty)\le \mathcal U_z(L)/\mathcal U_z(\infty)$. By
Theorem~\ref{thm:channel-duality}(ii), summed over $k\le L$,
$\mathcal V_z(L)=(p/q)^{a-z}\,\mathcal U_w(L)$ for every $L$, including $L=\infty$.
Therefore
\begin{equation}
\label{eq:cycle-dominance}
\mathcal U_z(L)\,\mathcal V_z(\infty)\ \ge\ \mathcal V_z(L)\,\mathcal U_z(\infty)
\qquad\text{for every }L\ge1 .
\end{equation}
Now condition on the reset times, and let $L_j$ be the number of steps
of the walk in the $j$-th cycle between resets ($L_j=\infty$ for the
last cycle if there are finitely many resets). As in the proof of
Theorem~\ref{thm:protocol}, the conditional probabilities of ruin and
of success are $\sum_j W_j\,\mathcal U_z(L_j)$ and $\sum_j W_j\,\mathcal V_z(L_j)$, with
the same weights $W_j\ge0$, the probability that the first $j-1$
cycles end in a reset. By~\eqref{eq:cycle-dominance}, the first times $\mathcal V_z(\infty)$ is at least the second times $\mathcal U_z(\infty)$, and
averaging over the reset times gives
$q_z\,(1-q^{(0)}_z)\ge(1-q_z)\,q^{(0)}_z$, that is $q_z\ge q^{(0)}_z$.
\end{proof}

\begin{remark}[Monotonicity in the intensity of resetting]
\label{rem:monotonicity}
Theorem~\ref{thm:direction} compares any resetting mechanism with no
resetting. A finer question is whether, within a family of mechanisms,
resetting more intensely always moves $q_z$ further in the same
direction; for geometric and Poissonian resetting it does
(Theorems~\ref{thm:main_result} and~\ref{thm:cont-uniqueness}). For
general mechanisms the answer depends on the parity of $a$. Under
periodic resetting every $T$ steps all cycles have $L=T-1$ steps, so
$q_z=\mathcal U_z(L)/[\mathcal U_z(L)+\mathcal V_z(L)]$, and monotonicity in $T$ is the monotonicity in $L$ of the ratio $\mathcal U_z(L)/\mathcal V_z(L)$. For odd $a$ it
fails: for $a=7$, $p=1/2$ and $z=2$, $q_z$ takes the values $0.923$,
$0.935$, $0.866$, $0.878$ for $T=6,7,8,9$, so that resetting less often
can increase the ruin probability. For even $a$ the decomposition used
in the proof of Theorem~\ref{thm:direction} shows why the situation is
different: the time $j$ to go from $a-z$ to $z$ is then even, and the
ratio is monotone as soon as $m\mapsto \mathcal U_z(z+2m)$ is log-concave, a
property that holds in every case we have computed ($a\le14$,
$p\in\{0.1,0.3,0.5,0.7,0.9\}$, all $z$). For odd $a$ that time is odd,
the decomposition mixes the two parity classes of the lattice, and the
example above shows that monotonicity is lost. We have not proved the
log-concavity, and leave it, together with its consequences for general
families of resetting mechanisms, to future work.
\end{remark}

\section{Geometric resetting: the ruin probability and the structure of the response}
\label{sec:invariance}

Section~\ref{sec:fixed-point} fixed the ruin probability at the
midpoint for every resetting mechanism, and Theorem~\ref{thm:direction}
fixed the direction of the effect on either side. For a mechanism with
a parameter one can ask more: how the ruin probability responds to a
change of that parameter. We answer this for geometric resetting, the
natural discrete-time protocol, for which the ruin probability is
available in closed form.

\subsection{Geometric resetting and the renewal identity}
\label{sec:model-renewal}

\begin{definition}[Geometric resetting]
\label{def:geometric}
Geometric resetting with parameter $\gamma\in(0,1)$ is the resetting
mechanism of Definition~\ref{def:process} in which each time step is,
independently of everything else, a reset step with probability
$\gamma$; equivalently, the times between successive resets are i.i.d.\
and geometric with parameter $\gamma$. We write $q_z(\gamma)$ for its
ruin probability.
\end{definition}

Let $R$ be the time of the first reset, with
$\mathbb{P}(R=k)=\gamma(1-\gamma)^{k-1}$. For any $\gamma\in(0,1)$,
$\tau$ is finite almost surely: in each regenerative cycle there is a
strictly positive probability
\[
\alpha_z(\gamma) := \mathbb{P}_z(T<R) = \sum_{k=1}^{\infty} s_{z,k}\,(1-\gamma)^k
\;>\;0
\]
that the walk reaches a boundary before the next reset, successive
cycles are probabilistically identical, and the probability of
surviving $n$ complete cycles without absorption is
$(1-\alpha_z(\gamma))^n\to0$ as $n\to\infty$.

To derive $q_z(\gamma)$, condition on whether absorption occurs before
the first reset. Since $\mathbb{P}(R>k)=(1-\gamma)^k$,
\[
q_z(\gamma) = \sum_{k=1}^{\infty} u_{z,k}\,(1-\gamma)^k
+ q_z(\gamma)\left(1-\sum_{k=1}^{\infty} s_{z,k}\,(1-\gamma)^k\right):
\]
the first term is ruin before any reset, the second is the event that
a reset occurs first, after which the walker restarts from $z$ with the
same ruin probability. Solving for $q_z(\gamma)$,
\begin{equation}
\label{eq:qz_generating}
q_z(\gamma) = \frac{U_z(1-\gamma)}{S_z(1-\gamma)}
=\frac{\mathbb{E}_z\!\left[(1-\gamma)^{T_0}\mathbf{1}_{\{T_0<T_a\}}\right]}
{\mathbb{E}_z\!\left[(1-\gamma)^{T}\right]}.
\end{equation}
We write $u(z;\gamma):=U_z(1-\gamma)$ and $v(z;\gamma):=V_z(1-\gamma)$
for the probabilities of absorption at $0$ and at $a$ before the first
reset, so that $q_z(\gamma)=u(z;\gamma)/[u(z;\gamma)+v(z;\gamma)]$.
Geometric resetting thus acts as exponential time-discounting of the
classical first-passage probabilities: the spatial dynamics between
resets are unchanged, and only the temporal weighting of trajectories
is modified, through the discount factor $(1-\gamma)^{T}$.

\subsection{The spectral formula}
\label{sec:spectral-formula}

Inserting the spectral generating functions of
Section~\ref{sec:spectral} into~\eqref{eq:qz_generating}, and writing
\begin{equation}
\label{eq:g_nu}
g_\nu(\gamma) := \frac{1}{1-\lambda_\nu(1-\gamma)}
\end{equation}
and setting $s=1-\gamma$ in $q_z(\gamma)=U_z(1-\gamma)/S_z(1-\gamma)$,
the common prefactor $\sqrt{pq}\,(1-\gamma)$ --- mode-independent, hence
shared by every term of both sums --- cancels between numerator and
denominator.

\begin{theorem}[Spectral formula for the ruin probability]
\label{thm:main_spectral}
For $\gamma\in(0,1)$,
\begin{equation}
\label{eq:qz_spectral}
q_z(\gamma) = \frac{\displaystyle\sum_{\nu=1}^{a-1} A_\nu(z)\,g_\nu(\gamma)}
{\displaystyle\sum_{\nu=1}^{a-1} \big[A_\nu(z)+B_\nu(z)\big]\,g_\nu(\gamma)},
\end{equation}
with $g_\nu(\gamma)$, $\lambda_\nu$, $A_\nu(z)$, and $B_\nu(z)$ given by
\eqref{eq:g_nu}, \eqref{eq:eigenvalues}, and \eqref{eq:A_nu}.
\end{theorem}
\begin{proof}
Immediate from the computation above.
\end{proof}

\begin{remark}[$\gamma=1$ is not resolved by inspection]
\label{rem:gamma1-continuity}
Setting $\gamma=1$ termwise in \eqref{eq:qz_spectral} does \emph{not}
remove the singularity: each $g_\nu(1)=1$ is finite, but
$\sum_\nu A_\nu(z)$ and $\sum_\nu[A_\nu(z)+B_\nu(z)]$ vanish
individually for every interior $z$ other than the two sites adjacent
to a boundary, since $\sqrt{pq}\sum_\nu A_\nu(z)=u_{z,1}$ and
$\sqrt{pq}\sum_\nu B_\nu(z)=v_{z,1}$ by \eqref{eq:u_spectral}, and a
walk cannot reach $0$ (respectively $a$) in a single step unless it
starts at $z=1$ (respectively $z=a-1$). So \eqref{eq:qz_spectral} is a
genuine $0/0$ form at $\gamma=1$ for generic $z$, including $z=a/2$,
and the limit $\gamma\to1^-$ must be extracted from the leading order
in $1-\gamma$ rather than by evaluating $g_\nu$ at $1$. This limit is
carried out rigorously, by a different route, in
Appendix~\ref{app:complete_proof} (Proposition~\ref{prop:limit-gamma-1}),
which gives $q_z(\gamma)\to1$ for $z<a/2$, $q_z(\gamma)\to0$
for $z>a/2$, and $q_z(\gamma)\to C^{\ast}$ at $z=a/2$ when $a$ is even,
$C^\ast$ being the reset-free ruin probability at the midpoint
(Theorem~\ref{thm:protocol} and Corollary~\ref{cor:classical}).
\end{remark}

Equation~\eqref{eq:qz_spectral} is the closed form of the ruin
probability under geometric resetting used throughout the rest of the
paper. Its numerical validation is carried out in
Section~\ref{sec:numerics}, with the full-precision sweeps collected in
Appendix~\ref{app:tables}.

\subsection{The sign change}
\label{sec:invariance-main}

Away from the midpoint, resetting does change the ruin probability, and
the closed form allows one to ask how. Is there a position separating
two regimes of resetting's influence, where the derivative
$\partial q_z(\gamma)/\partial\gamma$ changes sign, and is that position
unique? First-passage properties have been studied
extensively~\cite{Bray2013}, and resetting is known to alter them
systematically in unbounded domains~\cite{Evans2011, Pal2017}. For
absorption \emph{probabilities} in finite intervals, exact expressions
for splitting probabilities under resetting are
available~\cite{PalPrasad2019, ChenLiHuang2022}, but the structural
dependence of these probabilities on the resetting, and in particular
the existence and uniqueness of a reset-neutral starting point, has
received comparatively little attention. Classical ruin
theory~\cite{Feller1968, Redner2001} characterizes $q^{(0)}_z$. Work on stochastic resetting~\cite{Reuveni2016, Pal2017}
has focused mainly on optimizing first-passage times, and on their
fluctuations at the optimum, as well as on search
optimization~\cite{Evans2011, Kusmierz2014}.

As we show below, the answer is governed by the same reflection
symmetry $z\mapsto a-z$: the response
$h_z(\gamma):=\partial_\gamma q_z(\gamma)$ changes sign \emph{exactly
once}, at the location $z^{\dagger}:=a/2$, for every $p\in(0,1)$ and
every $\gamma\in(0,1)$. When $a$ is
even, this is a genuine zero of $h$ at the lattice site $z=z^{\dagger}$,
where the ruin probability is exactly invariant
(Corollary~\ref{cor:geometric-midpoint}). When $a$ is odd, $h_z$ is defined only
on integer sites, so $z=a/2$ is not itself among them; the sign change
is nonetheless geometrically pinned at the midpoint, always falling
between the two lattice sites bracketing it, $(a-1)/2$ and $(a+1)/2$,
independently of $p$ and $\gamma$. Section~\ref{sec:invariance-parity}
shows that this geometric statement becomes an exact invariance at
$z=a/2$ once $q_z(\gamma)$ is extended to real $z$ by analytic
continuation. What depends on $p$, in the odd case, is only the
relative size of the response at the two bracketing sites, not where
the sign change falls.

\begin{theorem}[Structure of the response under geometric resetting]
\label{thm:main_result}
Consider the walk of Definition~\ref{def:process} on $\{0,1,\dots,a\}$
with $p,q\in(0,1)$ ($p+q=1$) and geometric resetting
(Definition~\ref{def:geometric}) with parameter $\gamma\in(0,1)$, and write $h_z(\gamma):=\partial_\gamma q_z(\gamma)$
for $z\in\{1,\dots,a-1\}$.
\begin{enumerate}[label=(\roman*)]
\item \textbf{Existence and uniqueness of the sign change:}
For every $p\in(0,1)$ and every $\gamma\in(0,1)$,
\[
h_z(\gamma)>0 \ \text{for } z<a/2, \qquad
h_z(\gamma)<0 \ \text{for } z>a/2 .
\]
The sequence $h_1(\gamma),\dots,h_{a-1}(\gamma)$ therefore changes sign
exactly once: at $z=a/2$ when $a$ is even, and between the two
consecutive integers $(a-1)/2$ and $(a+1)/2$ when $a$ is odd.

\item \textbf{The sign change is geometrically pinned at the midpoint
(odd $a$):} If $a$ is odd, $h_z(\gamma)$ is defined only on the
lattice, so no site realizes $z=a/2$ exactly. Nonetheless the unique
sign transition is geometrically pinned at the midpoint, in the sense
that it occurs between the two lattice sites bracketing $a/2$, namely
$z=(a-1)/2$ and $z=(a+1)/2$, for \emph{every} $p\in(0,1)$ and
\emph{every} $\gamma\in(0,1)$. This pinning does not move with the
bias: what depends on $p$ is only the relative size of the response at
the two bracketing sites, not the location of the transition itself
(part (iii) below). Section~\ref{sec:invariance-parity} upgrades this
geometric statement to an exact invariance at $z=a/2$ under the analytic
continuation of $q_z(\gamma)$.

\item \textbf{Central deviation for odd $a$ does not vanish with $a$:}
When $a$ is odd, the two central sites carry equal and opposite
logarithmic deviations from the invariant value $C^\ast$ of
Theorem~\ref{thm:protocol},
$\delta_{(a-1)/2}(\gamma) = -\delta_{(a+1)/2}(\gamma)$, where
$\delta_z(\gamma):=\log\bigl[q_z(\gamma)/(1-q_z(\gamma))\bigr]
-\log\bigl[C^\ast/(1-C^\ast)\bigr]$. This common magnitude does
\emph{not} decay as $a\to\infty$ with $p,\gamma$ fixed: it converges to
the explicit limit
\[
\bigl|\delta_{(a\mp1)/2}(\gamma)\bigr|
\ \longrightarrow\ \tfrac12\log(r_+/r_-)\;>\;0
\]
as $a\to\infty$ with $p$ and $\gamma$ fixed,
where $r_+>r_->0$ are the characteristic roots of
Appendix~\ref{app:complete_proof} (Lemma~\ref{lem:roots}). Consequently
the two central sites are not asymptotically invariant. For
$p\neq1/2$, $C^\ast$ tends to $0$ or to $1$ as $a$ grows, which
compresses the corresponding difference $q_z-C^\ast$ even though the
logarithmic deviation remains bounded away from zero; in the unbiased
case $p=1/2$, $C^\ast=1/2$ for every $a$, and this compression does not
occur.
\end{enumerate}
\end{theorem}

\begin{proof}
The argument has two ingredients: the temporal representation of the
response (Theorem~\ref{thm:temporal} below), which turns the sign of
$h_z$ into a comparison of two conditioned, discounted absorption
times; and the reflection duality relating those two times, which
forces them to coincide exactly where $z=a-z$. Combined with the exact
antisymmetry $\delta_z(\gamma)=-\delta_{a-z}(\gamma)$ of
Corollary~\ref{cor:antisymmetry}, this gives all three parts. The
detailed derivation is given in Appendix~\ref{app:complete_proof}: part
(i) is Theorem~\ref{thm:temporal} together with Corollary~\ref{cor:unique};
parts (ii) and (iii) are, respectively, Corollary~\ref{cor:unique} and
Corollary~\ref{cor:central}.
\end{proof}

\begin{remark}[Physical interpretation]
\label{rem:critical-interpretation}
Resetting repeatedly interrupts a trajectory and restarts it from $z$.
By Theorem~\ref{thm:temporal} below, whether this favors ruin is
decided by which conditioned absorption channel is temporally favored
under the corresponding discounted law: resetting \emph{increases} the
ruin probability from $z$ exactly when the origin, conditioned on being
reached at all, is reached sooner (in the discounted-mean sense of
Section~\ref{sec:invariance-temporal}) than the far boundary is,
conditioned on being reached at all --- and \emph{decreases} it in the
opposite case. When $a$ is even, $z=a/2$ is the fixed point of the
reflection $z\leftrightarrow a-z$, and the exact antisymmetry of
$\delta_z$ under this reflection (Corollary~\ref{cor:antisymmetry})
makes the two conditioned times coincide there for every $p$, so that resetting has no effect at the midpoint, whatever the bias, in agreement with
Corollary~\ref{cor:geometric-midpoint}. When
$a$ is odd, this lattice-level protection is absent --- no site sits
exactly at the fixed point --- but the underlying symmetry is unaffected
by the parity of $a$, and it is what continues to pin the sign change
at $a/2$ itself (part (ii)): what the bias changes is only how
unequally the two conditioned times differ at the two neighboring
sites, not where they cross. This is the sense in which the location of the sign change is
independent of the bias, even though the magnitudes on either side
are not.
\end{remark}

\subsubsection{Numerical verification}
\label{sec:invariance-numerics}

Figure~\ref{fig:critical_point} displays $h_z(\gamma)$ for $a=10$
(even), and Figure~\ref{fig:critical_point_odd} for $a=11$ (odd), each
for several values of $p$ and $\gamma$, computed directly from the
closed form $u(z;\gamma),v(z;\gamma)$ of
Appendix~\ref{app:complete_proof}, evaluated independently of the
numerical evaluation of the spectral sum of Section~\ref{sec:spectral}
and so providing an independent computational cross-check of that
representation.

\begin{figure*}[!t]
\centering
\begin{tikzpicture}
\begin{axis}[
    width=0.92\textwidth,
    height=0.36\textwidth,
    xlabel={Initial position $z$},
    ylabel={$\displaystyle\frac{\partial q_z(\gamma)}{\partial \gamma}$},
    xlabel style={font=\footnotesize},
    ylabel style={font=\footnotesize},
    xmin=0, xmax=10,
    ymin=-1.2, ymax=1.62,
    xtick={0,1,...,10},
    grid=major,
    grid style={dashed,gray!30},
    legend pos=north east,
    legend style={
        font=\small,
        fill=white,
        fill opacity=0.9,
        draw=black,
        text opacity=1,
        cells={anchor=west},
        rounded corners=2pt
    },
    every axis plot/.append style={thick},
    tick label style={font=\footnotesize},
    clip=false,
]

\addplot[black, very thin, dashed, forget plot] coordinates {(0,0) (10,0)};

\addplot[blue, solid, mark=*, mark size=1.5pt, line width=1.1pt] coordinates {
(1, 0.009784) (2, 0.052609) (3, 0.208260) (4, 0.488381) (5, 0.000000)
(6, -0.488381) (7, -0.208260) (8, -0.052609) (9, -0.009784)
};
\addlegendentry{$p=0.5, \gamma=0.3$}

\addplot[blue!60!cyan, solid, mark=square*, mark size=1.5pt, line width=1.1pt] coordinates {
(1, 0.000074) (2, 0.001348) (3, 0.020621) (4, 0.218214) (5, 0.000000)
(6, -0.218214) (7, -0.020621) (8, -0.001348) (9, -0.000074)
};
\addlegendentry{$p=0.5, \gamma=0.6$}

\addplot[blue!30!cyan, solid, mark=triangle*, mark size=1.8pt, line width=1.1pt] coordinates {
(1, 0.000000) (2, 0.000001) (3, 0.000254) (4, 0.050252) (5, 0.000000)
(6, -0.050252) (7, -0.000254) (8, -0.000001) (9, -0.000000)
};
\addlegendentry{$p=0.5, \gamma=0.9$}

\addplot[green!60!black, densely dashed, mark=*, mark size=1.5pt, line width=1.2pt] coordinates {
(1, 0.001024) (2, 0.005813) (3, 0.025262) (4, 0.077963) (5, 0.000000)
(6, -0.971389) (7, -1.039181) (8, -0.316887) (9, -0.058591)
};
\addlegendentry{$p=0.4, \gamma=0.3$}

\addplot[red, dashed, mark=*, mark size=1.5pt, line width=1.2pt] coordinates {
(1, 0.058591) (2, 0.316887) (3, 1.039181) (4, 0.971389) (5, 0.000000)
(6, -0.077963) (7, -0.025262) (8, -0.005813) (9, -0.001024)
};
\addlegendentry{$p=0.6, \gamma=0.3$}

\addplot[only marks, mark=o, mark size=5pt, black, very thick, fill=white, forget plot]
coordinates {(5.0, 0.000)};

\draw[black, dashed, thick, opacity=0.7] (axis cs:5,-1.2) -- (axis cs:5,1.35);

\node[black, font=\small, fill=white, fill opacity=0.95, draw=black,
rounded corners, inner sep=4pt] at (axis cs:5, 1.45) {$z^\dagger = a/2 = 5$};

\fill[blue!8, opacity=0.3] (axis cs:0,-1.2) rectangle (axis cs:5,1.35);
\fill[red!8, opacity=0.3] (axis cs:5,-1.2) rectangle (axis cs:10,1.35);

\node[font=\small, align=center, text width=0.20\textwidth, fill=white,
fill opacity=0.7, rounded corners] at (axis cs:2.3, -0.72) {
\textbf{$z < z^\dagger$}\\[2pt] Reset \textcolor{blue}{increases} $q_z$\\
(for all $p \in (0,1)$)
};
\node[font=\small, align=center, text width=0.20\textwidth, fill=white,
fill opacity=0.7, rounded corners] at (axis cs:7.6, -0.72) {
\textbf{$z > z^\dagger$}\\[2pt] Reset \textcolor{blue}{decreases} $q_z$\\
(for all $p \in (0,1)$)
};

\node[font=\scriptsize, align=center, text width=0.19\textwidth, fill=yellow!20,
draw=orange, thick, rounded corners, inner sep=3pt]
at (axis cs:1.72, 1.28) {
\textbf{Sign change at $z=a/2$}\\
$a$ even: $z=5$ is a site,\\
so $\partial_\gamma q_5\equiv0$ exactly
};

\end{axis}
\end{tikzpicture}

\caption{Derivative $\partial q_z(\gamma) / \partial \gamma$ as a
function of initial position $z$ for $a = 10$. All curves are
analytic, computed from the closed form of
Appendix~\ref{app:complete_proof}. The critical point
$z^\dagger = 5 = a/2$ (circle) is universal: independent of both $p$
and $\gamma$. For all $p\in(0,1)$, resetting increases the ruin
probability for $z<5$ and decreases it for $z>5$. The symmetric case
$p=0.5$ (blue) is perfectly antisymmetric about $z=5$; the asymmetric
cases $p=0.4$ (green) and $p=0.6$ (red) show that bias changes the
magnitude of the response away from the midpoint, but never its
zero-crossing.}
\label{fig:critical_point}
\end{figure*}

Key observations for even $a$: all five curves cross zero exactly at
$z=5=a/2$, independently of $p$ and $\gamma$, confirming
Corollary~\ref{cor:geometric-midpoint};
the sign pattern of part (i) holds throughout; and, although the
critical point stays fixed, the magnitude of the response away from
the midpoint is strongly modulated by both the bias and the resetting
parameter, without ever displacing the crossing.

For comparison, Figure~\ref{fig:critical_point_odd} shows the odd case
$a=11$.

\begin{figure*}[!t]
\centering
\begin{tikzpicture}
\begin{axis}[
width=0.92\textwidth,
height=0.36\textwidth,
xlabel={Initial position $z$},
ylabel={$\displaystyle\frac{\partial q_z(\gamma)}{\partial \gamma}$},
xlabel style={font=\footnotesize},
ylabel style={font=\footnotesize},
xmin=0, xmax=11,
ymin=-1.5, ymax=1.95,
xtick={0,1,2,3,4,5,6,7,8,9,10,11},
grid=major,
grid style={dashed,gray!30},
legend pos=north east,
legend style={
font=\small,
fill=white,
fill opacity=0.9,
draw=black,
text opacity=1,
cells={anchor=west},
rounded corners=2pt
},
every axis plot/.append style={thick},
tick label style={font=\footnotesize},
clip=false,
]
\addplot[black, very thin, dashed, forget plot] coordinates {(0,0) (11,0)};
\addplot[blue, solid, mark=*, mark size=1.5pt, line width=1.1pt] coordinates {
(1, 0.004525)
(2, 0.025268)
(3, 0.109932)
(4, 0.357260)
(5, 0.411404)
(6, -0.411404)
(7, -0.357260)
(8, -0.109932)
(9, -0.025268)
(10, -0.004525)
};
\addlegendentry{$p=0.5, \gamma=0.3$}
\addplot[blue!60!cyan, solid, mark=square*, mark size=1.5pt, line width=1.1pt] coordinates {
(1, 0.000017)
(2, 0.000328)
(3, 0.005396)
(4, 0.073063)
(5, 0.389674)
(6, -0.389674)
(7, -0.073063)
(8, -0.005396)
(9, -0.000328)
(10, -0.000017)
};
\addlegendentry{$p=0.5, \gamma=0.6$}
\addplot[blue!30!cyan, solid, mark=triangle*, mark size=1.8pt, line width=1.1pt] coordinates {
(1, 0.000000)
(2, 0.000000)
(3, 0.000016)
(4, 0.003796)
(5, 0.456835)
(6, -0.456835)
(7, -0.003796)
(8, -0.000016)
(9, -0.000000)
(10, -0.000000)
};
\addlegendentry{$p=0.5, \gamma=0.9$}
\addplot[green!60!black, densely dashed, mark=*, mark size=1.5pt, line width=1.1pt] coordinates {
(1, 0.000375)
(2, 0.002205)
(3, 0.010292)
(4, 0.038996)
(5, 0.077002)
(6, -0.328968)
(7, -1.366904)
(8, -0.749066)
(9, -0.185517)
(10, -0.032352)
};
\addlegendentry{$p=0.4, \gamma=0.3$}
\addplot[red, dashed, mark=*, mark size=1.5pt, line width=1.1pt] coordinates {
(1, 0.032352)
(2, 0.185517)
(3, 0.749066)
(4, 1.366904)
(5, 0.328968)
(6, -0.077002)
(7, -0.038996)
(8, -0.010292)
(9, -0.002205)
(10, -0.000375)
};
\addlegendentry{$p=0.6, \gamma=0.3$}
\draw[orange, dashed, thick, opacity=0.4] (axis cs:5.5,-1.5) -- (axis cs:5.5,1.5);
\node[orange, font=\small, fill=white, fill opacity=0.8, draw=orange,
rounded corners, inner sep=3pt]
at (axis cs:5.5, 1.35) {$a/2 = 5.5$};
\fill[blue!8, opacity=0.3] (axis cs:0,-1.5) rectangle (axis cs:5.5,1.5);
\fill[red!8, opacity=0.3] (axis cs:5.5,-1.5) rectangle (axis cs:11,1.5);
\node[font=\small, align=center, text width=0.20\textwidth, fill=white,
fill opacity=0.7, rounded corners]
at (axis cs:2.3, -0.80) {
\textbf{$z < 5.5$}\\[2pt]
Reset \textcolor{blue}{increases} $q_z$
};
\node[font=\small, align=center, text width=0.20\textwidth, fill=white,
fill opacity=0.7, rounded corners]
at (axis cs:8.4, -0.80) {
\textbf{$z > 5.5$}\\[2pt]
Reset \textcolor{blue}{decreases} $q_z$
};
\node[font=\scriptsize, align=center, text width=0.19\textwidth, fill=yellow!20,
draw=orange, thick, rounded corners, inner sep=3pt]
at (axis cs:1.95, 1.42) {
\textbf{Sign change at $z=a/2$}\\
$a$ odd: $a/2=5.5$ is not a site,\\
so no start is neutral
};
\end{axis}
\end{tikzpicture}\caption{Derivative $\partial q_z(\gamma) / \partial \gamma$ as a
function of initial position $z$ for $a = 11$ (odd). All curves are
analytic, computed from the closed form of
Appendix~\ref{app:complete_proof}. No integer site is
invariant, and $a/2=5.5$ is not a lattice position. Theorem~\ref{thm:main_result}(ii)
guarantees that the sign transition always lies between $z=5$ and
$z=6$, for every $p\in(0,1)$ and $\gamma\in(0,1)$; the curves shown
illustrate this universal result. Bias changes only how unequally the
response is shared between the two central sites $z=5$ and $z=6$ ---
compare the green ($p=0.4$) and red ($p=0.6$) curves, which are mirror
images of each other under $z\mapsto11-z$ --- not the location of the
crossing itself.}
\label{fig:critical_point_odd}
\end{figure*}

Key observations for odd $a$: as an additional numerical check beyond
the five curves shown, we swept $p\in\{0.05,0.10,\dots,0.95\}$ and
$\gamma\in\{0.05,0.3,0.6,0.9\}$ at $a=11$ (76 combinations); the sign
change did not leave the pair $\{5,6\}$ in any of them, illustrating
--- without in any way substituting for --- the proof of part (ii). At
$p=0.5$, reflection symmetry gives $h_5=-h_6$. Under bias, this
equality generally disappears at the level of the response $h_z$
itself, since $h_z$ is the derivative of the nonlinear map
$\delta\mapsto q$, not of $\delta$. What remains exactly antisymmetric
is instead the logarithmic deviation, $\delta_5=-\delta_6$
(Theorem~\ref{thm:main_result}(iii)). Moreover, for the parameter values
illustrated here, this central logarithmic deviation \emph{increases}
rather than decreases as $\gamma$ approaches one: for $p=0.4$,
$|\delta_5|=|\delta_6|$ equals $0.924$ at $\gamma=0.3$, $1.589$ at
$\gamma=0.6$, and $3.014$ at $\gamma=0.9$ --- a distinct observation
from, though consistent with, the $a\to\infty$ statement of part (iii).

\subsection{The temporal criterion}
\label{sec:invariance-temporal}

The sign of the response has a direct probabilistic reading, in terms
of the mean absorption time \emph{conditioned on which boundary is
reached}, discounted by the same geometric weight that discounts the
ruin probability itself. Precisely: let $T$ be the absorption time
of the walk \emph{without} resetting started at $z$ (as in
Section~\ref{sec:model-generating}), let $A_0,A_a$
denote the (disjoint) events of absorption at $0$, at $a$, and for $t=1-\gamma$ set the conditioned discounted means
\[
\begin{aligned}
&M_0(z):=\frac{\mathbb E_z\bigl[T\,t^{T};A_0\bigr]}
             {\mathbb E_z\bigl[t^{T};A_0\bigr]}
=t\,\partial_t\log u(z;\gamma),
\qquad\\
&M_a(z):=\frac{\mathbb E_z\bigl[T\,t^{T};A_a\bigr]}
             {\mathbb E_z\bigl[t^{T};A_a\bigr]}
=t\,\partial_t\log v(z;\gamma).
\end{aligned}
\]
Both are means of $T$ under the exponentially tilted (discounted)
law with weight $t^{T}$, conditioned on absorption at the
respective boundary; we refer to them as the \emph{discounted mean
absorption times} at $0$ and at $a$.

\begin{theorem}[Temporal criterion for the response]
\label{thm:temporal}
For every $z\in\{1,\dots,a-1\}$ and every $\gamma\in(0,1)$,
\[
\operatorname{sgn}h_z(\gamma)
=\operatorname{sgn}\bigl(M_a(z)-M_0(z)\bigr)
=\operatorname{sgn}\Bigl(\frac a2-z\Bigr).
\]
Equivalently: resetting increases the ruin probability from $z$ if and
only if $M_0(z)<M_a(z)$, that is, if and only if the origin, conditioned
on being reached, is reached sooner in the discounted-mean sense than
the far boundary, conditioned on being reached.
\end{theorem}

The mechanism behind Theorem~\ref{thm:temporal} is short to state. The
logarithmic derivative of a ratio of two discounted functionals is the
difference of their own discounted mean absorption times --- a purely
generating-function identity, differentiating $\log(U/V)$ termwise
inside the disc of convergence, that uses neither a state space nor the
Markov property. Applied to $U=u(z;\gamma)$, $V=v(z;\gamma)$, it gives
$\operatorname{sgn}h_z=\operatorname{sgn}(M_a(z)-M_0(z))$. A duality
identity converts $M_a(z)$ into $M_0(a-z)$: the mean time to reach $a$
from $z$, conditioned on doing so, equals the mean time to reach $0$
from the mirror site $a-z$, conditioned on doing so, because $v(z;\gamma)$
and $u(a-z;\gamma)$ differ only by a $\gamma$-independent constant.
The claim then reduces to $M_0(z)$ being strictly increasing in $z$
--- itself immediate from the strong Markov property applied at
the first visit to an interior site --- so that $M_0(a-z)-M_0(z)$
changes sign exactly where $a-z=z$. The full chain of lemmas is given
in Appendix~\ref{app:complete_proof} (Proposition~\ref{prop:gf-identity},
Lemma~\ref{lem:duality-times}, Lemma~\ref{lem:monotone-T}).

\subsection{Remark: what the proof uses}
\label{sec:invariance-scope}

\begin{remark}
\label{rem:not-markov}
The generating-function step of the argument --- that
$\operatorname{sgn}h_z=\operatorname{sgn}(M_a(z)-M_0(z))$ --- is
independent of any Markov description of the resetting dynamics: it
uses only that absorption at $0$ and at $a$ are disjoint events, and
that the discount is geometric in the number of steps. To specialize
the resulting identity to the present nearest-neighbor walk, however,
we do use the strong Markov property of the walk itself: it is what
gives both the duality $M_a(z)=M_0(a-z)$ and the strict monotonicity of
$M_0(z)$. In short, what is Markov-free is the mechanism connecting
resetting's effect to a comparison of times; what is specific to this
walk --- and does use the strong Markov property --- is determining
\emph{which} time is larger.
\end{remark}

\begin{remark}[Scope: beyond a single reset site]
\label{rem:multisite-scope}
The mechanism isolated above is not specific to resetting to a single
site. When the walk is reset to one of several interior sites
according to a distribution $\pi$, the ruin probability is governed by
a single coupling constant
$C(\pi,\gamma)=\sum_z\pi_z\,u(z;\gamma)\big/\sum_z\pi_z\,[u(z;\gamma)+v(z;\gamma)]$~\cite{PaperII},
and the structure that controls it is precisely the channel
duality~\eqref{eq:channel-duality}: an involution $z\mapsto a-z$
exchanging the two channels up to the mode-independent weight
$(p/q)^{a-z}$. For a set of reset sites symmetric under $z\mapsto a-z$
there is a family of reset distributions $\pi^\ast$, with
$\pi^\ast_z/\pi^\ast_{a-z}=(q/p)^{a/2-z}$, for which
$C(\pi^\ast,\gamma)$ is independent of $\gamma$ and equals the midpoint
value $C^\ast$ of Theorem~\ref{thm:protocol}. Symmetry of the sites
alone does not suffice; the weights must compensate the bias. The
single-site invariance proved here is the case of $\pi$ concentrated
at $z=a/2$. The general statement, for absorbed Markov chains whose
spectral coefficients satisfy such a twisted duality, is proved in
Ref.~\cite{PaperII}; Figure~\ref{fig:multisite-separatrix} in
Appendix~\ref{app:multisite} illustrates it for two reset sites.
\end{remark}

\subsection{Parity}
\label{sec:invariance-parity}

The statement that the odd-$a$ sign change is geometrically pinned at
$a/2$ (Theorem~\ref{thm:main_result}(ii)) is not merely an
interpolation convention. For odd $a$, $a/2$ is not simply the midpoint
between the two lattice sites where the discrete sign transition
occurs: under the analytic continuation of $q_z(\gamma)$ developed
below, $a/2$ is the very same exact invariant point as in the even
case.

The invariance at $z=a/2$ is best understood as a property of the
reflection symmetry $z\mapsto a-z$ itself, valid for the analytic
continuation of $q_z(\gamma)$ to every real $z\in(0,a)$, and not
merely a coincidence restricted to even domain sizes: when $a$ is odd,
the same cancellation occurs exactly at $z=a/2$, but no lattice site
realizes it, so the invariance has no direct probabilistic meaning for
the discrete chain (Appendix~\ref{app:complete_proof},
Remark~\ref{rem:app-invariance-any-parity}). This is confirmed
independently by the diffusive limit of
Section~\ref{sec:continuum}, where the invariance holds
unconditionally and with no even--odd dichotomy, since a continuous
state space removes the lattice obstruction altogether: the discrete
and continuum results are the same statement about the same
reflection symmetry, evaluated on two different domains.

What distinguishes the two parities, then, is not whether $z=a/2$ is
special --- it always is --- but whether a trajectory can start there.
For even $a$ this is an observable fact about the discrete chain:
$q_{a/2}(\gamma)$ is a ruin probability, and it is exactly independent
of $\gamma$. For odd $a$ the two integer sites nearest the center,
$(a-1)/2$ and $(a+1)/2$, are what a real walker can occupy, and
Theorem~\ref{thm:main_result}(iii) shows that their deviation from
$C^\ast$ is equal and opposite but does \emph{not} vanish as $a$ grows:
for $p\neq1/2$, any apparent near-invariance of the two central sites at large $a$ is a statement about the compressed \emph{scale} of $q_z$ near the boundary values $0$ and $1$, not about
the underlying logarithmic deviation, which stays bounded away from
zero (Remark~\ref{rem:no-suppression}). Two facts about odd
$a$ deserve emphasis: the sign change sits at $a/2$ for every
drift, and the central deviation converges
to a positive constant rather than vanishing with $a$. For
$p\neq1/2$ what decreases with $a$ is $q_z-C^\ast$ itself, since
$C^\ast$ is compressed towards $0$ or $1$; the deviation measured on
the logarithmic scale does not. For $p=1/2$, $C^\ast=1/2$ and no such
compression occurs.

\section{The diffusive limit}
\label{sec:continuum}

The same questions can be asked in continuous time and space, and it
is natural to ask whether the midpoint invariance of
Section~\ref{sec:fixed-point} is a lattice effect or survives the
diffusive limit. We show here that it survives, exactly and for every
resetting mechanism, and that under Poissonian resetting --- the
continuous-time counterpart of geometric resetting --- so does the
uniqueness of the sign change, by the same mechanism identified in
Appendix~\ref{app:complete_proof}: the midpoint is the fixed point of a
reflection under which the deviation is antisymmetric. What does
\emph{not} survive is the parity effect, and we make that precise at
the end.

\subsection{The continuous problem}
\label{sec:continuum-problem}

Consider a diffusion on $[0,a]$ with constant drift $\mu$ and
volatility $\sigma$, absorbed at both endpoints, and returned to its
starting point at a set of reset times independent of the driving
Brownian motion: the continuous counterpart of
Definition~\ref{def:process}. We assume that the reset times have no
accumulation point in finite time, which excludes Zeno-type protocols
and is satisfied by all the usual ones. The times between resets may,
for instance, be i.i.d.\ with an arbitrary density; Poissonian resetting at
rate $\lambda$ is the case of exponential times. Let $q(x)$ denote the
probability of absorption at $0$ starting from $x$. As in the discrete
case, the invariance at the midpoint does not depend on the resetting
mechanism.

\begin{proposition}[Independence of the resetting mechanism in the continuum]
\label{prop:cont-protocol}
Let the diffusion with resetting start at $a/2$, with an arbitrary
resetting mechanism whose reset times have no accumulation point in
finite time, and suppose that absorption occurs almost surely.
Then $q(a/2)=1/(1+e^{\mu a/\sigma^{2}})$, whatever the law of the reset
times.
\end{proposition}

\begin{proof}
Let $g_0(t)$ and $g_a(t)$ be the densities of the exit time through
$0$ and through $a$ of the reset-free diffusion started at $a/2$. By
Girsanov's theorem, relative to the driftless diffusion with the same
volatility, the density of the path up to the exit time $T$ is
$\exp\bigl(\mu(X_T-a/2)/\sigma^{2}-\mu^{2}T/(2\sigma^{2})\bigr)$, with
$X_T-a/2=-a/2$ on exit through $0$ and $+a/2$ on exit through $a$.
Without drift the two exit densities coincide by the reflection
symmetry about $a/2$, and the time-dependent factor is common to both
channels; hence $g_0(t)=e^{-\mu a/\sigma^{2}}g_a(t)$ for every $t>0$,
the continuous counterpart of~\eqref{eq:midpoint-duality}. Conditioning
on the reset times and arguing cycle by cycle as in the proof of
Theorem~\ref{thm:protocol} gives the result.
\end{proof}

\subsection{Poissonian resetting}
\label{sec:continuum-poisson}

Under Poissonian resetting at rate $\lambda>0$, let $u(x)$ and $s(x)$
denote, respectively, the probability of absorption at $0$ before the
first reset and the probability of absorption at either endpoint before
the first reset, both started from $x$. Exactly as in the discrete
case, the renewal argument at the first reset gives
\begin{equation*}
q(x)=\frac{u(x)}{s(x)} ,
\end{equation*}
since the walker that is reset before absorption starts afresh from the
same point.

Because the reset clock is exponential of rate $\lambda$ and
independent of the diffusion, surviving to time $t$ without a reset
carries the factor $e^{-\lambda t}$; both $u$ and $s$ are therefore
$\lambda$-discounted absorption functionals and satisfy the same
ordinary differential equation
\begin{equation}
\label{eq:cont-ode}
\frac{\sigma^{2}}{2}\,w''+\mu\,w'-\lambda\,w=0,
\end{equation}
with boundary conditions
\[
u(0)=1,\quad u(a)=0,
\qquad
s(0)=s(a)=1 .
\]
The functionals $u$ and $s$ are $\lambda$-Laplace transforms of the
exit time through one or either endpoint, and splitting probabilities
for Brownian motion in an interval under Poissonian resetting have been
obtained previously~\cite{PalPrasad2019, Villarroel2022}. In
particular, the ruin probability $q(x)=u(x)/s(x)$ that follows from
Eq.~\eqref{eq:cont-solutions} below coincides, for unit volatility and
with the reset position equal to the starting point, with the exit
probability through the lower endpoint obtained by
Abundo~\cite{Abundo2025} for drifted Brownian motion with Poissonian
resetting; here it is derived independently, within the same renewal
framework as the discrete problem of Section~\ref{sec:invariance}. What the
present analysis adds is the observation that at $x=a/2$ the
dependence on $\lambda$ cancels identically
(Theorem~\ref{thm:cont-invariance}), and that this is the only point at
which the response vanishes (Theorem~\ref{thm:cont-uniqueness}).

\subsection{Solution and the invariance}
\label{sec:continuum-invariance}

\begin{theorem}[Midpoint invariance in the continuum]
\label{thm:cont-invariance}
Let $m:=\mu/\sigma^{2}$ and
$\beta:=\sqrt{2\lambda/\sigma^{2}+\mu^{2}/\sigma^{4}}$. Then
\begin{equation}
\label{eq:cont-solutions}
u(x)=e^{-mx}\,\frac{\sinh\beta(a-x)}{\sinh\beta a},\end{equation}
\begin{equation*}s(x)=e^{-mx}\,
\frac{\sinh\beta(a-x)+e^{ma}\sinh\beta x}{\sinh\beta a},
\end{equation*}
and consequently, at the midpoint,
\begin{equation}
\label{eq:cont-midpoint}
q(a/2)=\frac{1}{1+e^{\mu a/\sigma^{2}}},\end{equation}
\begin{equation*}\frac{\partial q(a/2)}{\partial\lambda}\equiv0 .
\end{equation*}
The ruin probability from the midpoint is therefore independent of the
resetting rate, for every drift and every volatility, in agreement with
Proposition~\ref{prop:cont-protocol}.
\end{theorem}

\begin{proof}
The substitution $w=e^{-mx}\,\widetilde w$ removes the first-order term
from \eqref{eq:cont-ode} and turns it into
$\widetilde w''=\beta^{2}\widetilde w$, whose solutions are combinations
of $\sinh\beta x$ and $\sinh\beta(a-x)$. Imposing the four boundary
conditions gives \eqref{eq:cont-solutions}; note that $e^{-mx}$ is
common to both, and that the boundary values of $s$ differ by the
factor $e^{ma}$ after the substitution, which is the origin of that
term.

At $x=a/2$ the two hyperbolic sines coincide,
$\sinh\beta(a-x)=\sinh\beta x=\sinh(\beta a/2)$, so both numerators
share this factor:
\[
q(a/2)=\frac{\sinh(\beta a/2)}
            {\sinh(\beta a/2)\,\bigl(1+e^{ma}\bigr)}
      =\frac{1}{1+e^{ma}} .
\]
The parameter $\lambda$ enters only through $\beta$, which cancels;
hence the derivative vanishes identically.
\end{proof}

\subsection{Uniqueness of the sign change}
\label{sec:continuum-uniqueness}
\label{subsec:cont-uniqueness}

The invariance established above locates one point at which the
response to resetting vanishes. We now show that it is the only one,
which is the continuous counterpart of Corollary~\ref{cor:unique}. It
is convenient to introduce the complementary functional
\begin{equation}
v(x) \;=\; s(x)-u(x)
\;=\; e^{-mx}\,e^{ma}\,\frac{\sinh\beta x}{\sinh\beta a},
\label{eq:cont-v}
\end{equation}
the discounted probability of absorption at $a$ before the first reset,
so that $q=u/(u+v)$.

\begin{lemma}[Duality in the diffusive limit]
\label{lem:cont-duality}
For every $x\in(0,a)$ and every $\lambda>0$,
\begin{equation}
v(x)\;=\;e^{2m(a-x)}\,u(a-x).
\label{eq:cont-duality}
\end{equation}
\end{lemma}

\begin{proof}
From \eqref{eq:cont-solutions} and \eqref{eq:cont-v},
\begin{equation*}
\begin{aligned}
u(a-x)=e^{-m(a-x)}\,\frac{\sinh\beta x}{\sinh\beta a},\\[2pt]
v(x)=e^{m(a-x)}\,\frac{\sinh\beta x}{\sinh\beta a}.
\end{aligned}
\end{equation*}
Dividing yields \eqref{eq:cont-duality}.
\end{proof}

Equation~\eqref{eq:cont-duality} is the exact analogue of the lattice
duality $B_\nu(z)=\kappa(z)A_\nu(a-z)$ with $\kappa(z)=(p/q)^{a-z}$:
under the correspondence $\log(p/q)\leftrightarrow 2\mu/\sigma^{2}$ the
two weights coincide. The two problems therefore carry the same
involution, and the argument below is the same one.

Following the lattice construction, set
\begin{equation}
\delta(x)\;:=\;\log\frac{u(x)}{v(x)}+ma
\;=\;\log\sinh\beta(a-x)-\log\sinh\beta x ,
\label{eq:cont-delta}
\end{equation}
so that
\begin{equation}
q(x)\;=\;\frac{1}{1+e^{\,ma-\delta(x)}} .
\label{eq:cont-q-delta}
\end{equation}

\begin{lemma}[Antisymmetry and monotonicity of $\delta$]
\label{lem:cont-delta}
For every $\lambda>0$ the function $\delta$ of \eqref{eq:cont-delta}
satisfies $\delta(a-x)=-\delta(x)$ on $(0,a)$, is strictly decreasing,
and vanishes only at $x=a/2$.
\end{lemma}

\begin{proof}
Antisymmetry is immediate from \eqref{eq:cont-delta} by exchanging $x$
and $a-x$. Differentiating,
$\delta'(x)=-\beta\coth\beta(a-x)-\beta\coth\beta x<0$ on $(0,a)$, since
$\coth$ is positive there. A strictly decreasing odd function about
$a/2$ vanishes only at $a/2$.
\end{proof}

\begin{theorem}[Uniqueness of the sign change in the diffusive limit]
\label{thm:cont-uniqueness}
Let $\mu\in\mathbb{R}$, $\sigma>0$ and $a>0$. Then for every $\lambda>0$
and every $x\in(0,a)$,
\begin{equation}
\operatorname{sgn}\frac{\partial q}{\partial\lambda}(x)
\;=\;\operatorname{sgn}\!\left(\frac{a}{2}-x\right).
\label{eq:cont-sign}
\end{equation}
In particular $\partial_\lambda q$ vanishes at $x=a/2$ and nowhere
else: resetting increases the ruin probability strictly on $(0,a/2)$
and decreases it strictly on $(a/2,a)$, for every drift and every
resetting rate.
\end{theorem}

\begin{proof}
Only $\delta$ depends on $\lambda$ in \eqref{eq:cont-q-delta}, and it
does so through $\beta$ alone. Differentiating \eqref{eq:cont-q-delta},
\begin{equation*}
\frac{\partial q}{\partial\lambda}
= q(1-q)\,\frac{\partial\delta}{\partial\beta}\,
  \frac{\partial\beta}{\partial\lambda},\end{equation*}
\begin{equation}\frac{\partial\beta}{\partial\lambda}=\frac{1}{\sigma^{2}\beta}>0 ,
\label{eq:cont-dq}
\end{equation}
and $q(1-q)>0$ on $(0,a)$. Hence $\partial_\lambda q$ has the sign
of $\partial_\beta\delta$. From \eqref{eq:cont-delta},
\begin{equation*}
\frac{\partial\delta}{\partial\beta}
=(a-x)\coth\beta(a-x)-x\coth\beta x
= \chi(a-x)-\chi(x),
\end{equation*}
where $\chi(y):=y\coth\beta y$.
The function $\chi$ is strictly increasing on $(0,\infty)$: writing
$\eta=\beta y$,
\begin{equation*}
\begin{aligned}
\chi'(y)&=\coth \eta-\frac{\eta}{\sinh^{2}\eta}
      =\frac{\sinh \eta\cosh \eta-\eta}{\sinh^{2}\eta}\\[2pt]
     &=\frac{\tfrac12\sinh 2\eta-\eta}{\sinh^{2}\eta}>0,
\end{aligned}
\end{equation*}
because $\sinh 2\eta>2\eta $ for $\eta>0$. Consequently $\chi(a-x)-\chi(x)$ is strictly
decreasing in $x$, and it is odd about $a/2$, so it is positive on
$(0,a/2)$, zero at $a/2$, and negative on $(a/2,a)$. With
\eqref{eq:cont-dq} this gives \eqref{eq:cont-sign}.
\end{proof}

\subsection{Correspondence with the discrete result}
\label{sec:continuum-correspondence}

The continuous value \eqref{eq:cont-midpoint} is the exact analogue of
the discrete one. Writing the discrete invariant as
\[
C^\ast=\frac{1}{1+(p/q)^{a/2}}
\]
and the continuous one as $1/(1+e^{\mu a/\sigma^{2}})$, the two agree
under the standard correspondence
$\log(p/q)\leftrightarrow 2\mu/\sigma^{2}$ between the lattice bias and
the drift, the half-length $a/2$ playing the same role in both. The
correspondence extends to the mechanism itself, not merely to the
invariant value: see the duality~\eqref{eq:cont-duality} and the
paragraph following its proof in
Section~\ref{sec:continuum-uniqueness}, the continuous analogue of the
discrete duality $B_\nu(z)=\kappa(z)A_\nu(a-z)$ underlying
Appendix~\ref{app:complete_proof}.

\begin{remark}[Scaling]
\label{rem:scaling}
The correspondence above is the one obtained under the diffusive
scaling. With lattice spacing $\varepsilon$ and time step
$\Delta t=\varepsilon^{2}/\sigma^{2}$, the choice
\begin{equation*}
p_\varepsilon=\frac12+\frac{\mu\varepsilon}{2\sigma^{2}}
\end{equation*}
gives drift $\varepsilon(2p_\varepsilon-1)/\Delta t=\mu$ and variance
rate $\varepsilon^{2}/\Delta t=\sigma^{2}$, so the interpolated walk
converges to the diffusion with drift $\mu$ and volatility $\sigma$,
while the reset probability per step and the Poisson rate are related by
\[
\gamma=\lambda\,\Delta t=\frac{\lambda\varepsilon^{2}}{\sigma^{2}} .
\]
The familiar form with time step $\varepsilon^{2}$ and
$\gamma\simeq\lambda\varepsilon^{2}$ is the special case $\sigma=1$.
We do not use this
limit in the proof of Theorem~\ref{thm:cont-invariance} or
Theorem~\ref{thm:cont-uniqueness}, both self-contained computations in
the continuous model; the scaling is recorded only to indicate that the
two problems are the two ends of one family.
\end{remark}

\subsection{What is intrinsically discrete}
\label{sec:continuum-parity}

\begin{remark}[Neither the invariance nor its uniqueness is a lattice
effect]
\label{rem:parity-continuum}
Both Theorem~\ref{thm:cont-invariance} and
Theorem~\ref{thm:cont-uniqueness} rest on the same mechanism as the
discrete problem (Appendix~\ref{app:complete_proof}): the two
absorption functionals are exchanged by the reflection $x\mapsto a-x$
up to an explicit weight, so their log-ratio $\delta$ is odd about the
fixed point $a/2$; monotonicity of $\delta$ in the resetting parameter
then forces the response to vanish there, and only there. Neither the
vanishing nor its uniqueness required anything specific to the lattice
or to the diffusion beyond this duality --- they are the same statement,
evaluated on two different state spaces.

What genuinely distinguishes the two regimes is a fact about geometry,
not about the mechanism. In the continuum the point $x=a/2$ always
exists, for every $a>0$, so the neutral position is always realized by
an actual trajectory. On the lattice it is realized as a site only when
$a$ is even; for odd $a$ the fixed point falls strictly between the two
central sites, $(a-1)/2$ and $(a+1)/2$, and
Theorem~\ref{thm:main_result}(ii) shows that the sign change is
nonetheless geometrically pinned there, without being realized exactly
by either site, while Theorem~\ref{thm:main_result}(iii) and
Corollary~\ref{cor:central} show that the residual deviation at those
two sites does \emph{not} vanish as $a\to\infty$. The parity effect is
therefore not a failure of the invariance itself, which is robust and
holds unconditionally under the analytic continuation of $q_z(\gamma)$
to real $z$ (Remark~\ref{rem:app-invariance-any-parity}); it is a
question, entirely settled by the lattice's geometry, of whether a
site happens to coincide with the fixed point of a symmetry that holds
regardless.
\end{remark}

Taken together, the two theorems of this section show that the
invariance is not a lattice phenomenon. It survives the passage to
continuous time and space, where the invariant position always exists
because no lattice constrains it, and where the sign change is again
unique. What the lattice contributes is only the question of whether a
site realizes that position.

\section{Numerical validation}
\label{sec:numerics}

Every quantity reported in this paper is evaluated through three
independent computational representations of the same process, with no
intermediate numerical output shared between them. The closed form and
the independent linear solve are compared over the full parameter
sweep, while the spectral representation is retained as an algebraic
and numerical cross-check in the regime where it is well conditioned.
A separate Monte Carlo simulation provides an independent statistical
check. This section states what is compared against what, reports the
level of agreement obtained, and records the two places where the
comparison has a genuine limit.

\subsection{Three independent computations}
\label{subsec:three-ways}

\emph{(i) Spectral representation.} The mode amplitudes $A_\nu(z)$ and
$B_\nu(z)$ of \eqref{eq:A_nu} and the paired formula immediately
following it are summed against the
weights $g_\nu(\gamma)=[1-\lambda_\nu(1-\gamma)]^{-1}$, as in
Section~\ref{sec:spectral}.

\emph{(ii) Closed form.} The characteristic roots $r_\pm$ of
Lemma~\ref{lem:roots} are formed and $q_z$ evaluated through
Corollary~\ref{cor:rho-closed}, in logarithmic variables so that no
intermediate quantity overflows for large $a$.

\emph{(iii) Direct linear solve.} The system
\begin{equation}
\begin{aligned}
f(x)&=\gamma\,f(z)+(1-\gamma)\bigl[p\,f(x+1)+q\,f(x-1)\bigr],\\
f(0)&=1,\qquad f(a)=0,
\end{aligned}
\label{eq:num-linear}
\end{equation}
is solved by Gaussian elimination with partial pivoting in
\texttt{long double} arithmetic, which carries about eighteen
significant digits, and $q_z(\gamma)=f(z)$. This route uses no spectral
decomposition, no generating function and no renewal identity: it is
the definition of the process written as a linear system, and it
serves as the reference against which the other two are judged.

To these is added a direct Monte Carlo simulation of the walk, which
plays the separate role described in
Section~\ref{subsec:montecarlo}.

Over a sweep of $a\in\{3,4,5,8,10,11,12,13\}$, $p\in\{0.2,0.35,0.5,0.6,0.85\}$,
$\gamma\in\{0.05,0.2,0.5,0.8,0.95\}$ and every interior site, the discrepancy between routes (ii) and (iii) remains at or below
$3\times10^{-11}$, and is at the level of roundoff over most of the
sweep. The worst case is attained at the corner $a=13$, $p=0.85$,
$\gamma=0.95$, where the linear system is at its most
ill-conditioned. The residual
reflects the conditioning of \eqref{eq:num-linear} rather than the
closed form. For $a=13$ and $p=0.85$ the spectral condition number of
that system rises from $3\times10^{1}$ at $\gamma=0.05$ to
$1\times10^{4}$ at $\gamma=0.5$ and $3\times10^{11}$ at
$\gamma=0.95$, and the observed discrepancy remains consistent with
$\kappa$ times the unit roundoff of the arithmetic used throughout
that range. Route (ii) evaluates the same quantity in a numerically stable form
and does not exhibit the corresponding loss of accuracy over the
tested range.

\subsection{Conditioning of the spectral sum}
\label{subsec:conditioning}

Route (i) is exact as an identity and is the one from which the
structure of the problem is read, but it is not the one to evaluate
in floating point. The amplitude $A_\nu(z)$ carries the factor
$(q/p)^{z/2}$ and $B_\nu(z)$ the factor $(p/q)^{(a-z)/2}$, so for
large $a$, strong bias and $\gamma$ close to one the individual modes
are large and of alternating sign while their sum is small. The
resulting cancellation is severe: at $a=13$, $p=0.85$ and
$\gamma=0.95$ the double-precision spectral sum departs from
\eqref{eq:num-linear} by $1.3\times10^{-4}$, whereas in the moderate
regime $a\le10$, $0.3\le p\le0.7$, $\gamma\le0.8$ the same comparison gives $8\times10^{-13}$.

All figures and tables in this paper are therefore evaluated through
route (ii), which groups the same quantity so as to avoid the severe
cancellation present in the spectral sum. Route (i) is retained as a cross-check in the moderate
regime, where it is reliable.

\begin{remark}[A removable singularity to be avoided]
\label{rem:num-pole}
The spectral sum admits an equivalent form in which the amplitudes
carry a factor $\lambda_\nu^{-1}$ and the weights a compensating
factor $\lambda_\nu(1-\gamma)$. That form is algebraically identical
but numerically hazardous: $\lambda_\nu=2\sqrt{pq}\cos(\pi\nu/a)$
vanishes at $\nu=a/2$ when $a$ is even, and the singularity there is
removable only in the product. Setting the amplitude to zero at that
mode, rather than taking the limit, silently discards one mode of the
expansion and corrupts the result at every odd site. The form used
here, Lemma~\ref{lem:bridge}, has no singularity: $\lambda_\nu=0$
gives simply $g_\nu=1$.
\end{remark}

\subsection{The invariance}
\label{subsec:num-invariance}

\begin{figure}[!tb]
\centering
\begin{tikzpicture}
\begin{axis}[papel,
  xlabel={starting site $z$}, ylabel={$q_z(\gamma)$},
  xmin=0.6, xmax=9.4, ymin=-0.04, ymax=1.06,
  xtick={1,...,9}, ytick={0,0.2,0.4,0.6,0.8,1},
  legend pos=north east, legend cell align=left,
  legend entries={{$\gamma=10^{-4}$},{$\gamma=0.2$},{$\gamma=0.4$},
                  {$\gamma=0.6$},{$\gamma=0.8$}},
]
\addplot[black!35, dashed, forget plot]
  coordinates {(5,-0.04) (5,1.06)};
\addplot[black!35, dashed, forget plot]
  coordinates {(0.6,0.116363636364) (9.4,0.116363636364)};
\addplot[black, mark=*, mark size=1.9pt] coordinates {(1,0.661285984) (2,0.435064014) (3,0.284115902) (4,0.183458008) (5,0.116363636) (6,0.071653799) (7,0.041865941) (8,0.022022288) (9,0.008804216)};
\addplot[blue!65!black, mark=square*, mark size=1.9pt] coordinates {(1,0.982903629) (2,0.915935288) (3,0.709264295) (4,0.360927662) (5,0.116363636) (6,0.029790832) (7,0.007058321) (8,0.001589080) (9,0.000301543)};
\addplot[teal, mark=triangle*, mark size=1.9pt] coordinates {(1,0.999155699) (2,0.991231569) (3,0.921996351) (4,0.554967886) (5,0.116363636) (6,0.013715546) (7,0.001464996) (8,0.000153380) (9,0.000014654)};
\addplot[orange!85!black, mark=diamond*, mark size=1.9pt] coordinates {(1,0.999978066) (2,0.999451938) (3,0.986988881) (4,0.759644466) (5,0.116363636) (6,0.005457010) (7,0.000228555) (8,0.000009509) (9,0.000000380)};
\addplot[red!75!black, mark=pentagon*, mark size=1.9pt] coordinates {(1,0.999999931) (2,0.999992878) (3,0.999272882) (4,0.930809052) (5,0.116363636) (6,0.001287409) (7,0.000012618) (8,0.000000124) (9,0.000000001)};
\addplot[only marks, mark=*, mark size=3.6pt, black, forget plot]
  coordinates {(5,0.116363636364)};
\end{axis}
\end{tikzpicture}
\caption{Ruin probability $q_z(\gamma)$ against the starting site, for
$a=10$ and $p=0.6$, computed analytically from the closed form; the
curve for $\gamma=10^{-4}$ is, at the scale of the figure,
indistinguishable from the classical ruin probability $q^{(0)}_z$. The
curves fan out over almost the whole unit
interval as the resetting probability grows, and meet at the single
point $z=a/2$, where $q_5(\gamma)=C^{\ast}=0.116363636364\ldots$ to
twelve decimal places for every $\gamma$ shown.}
\label{fig:num-crossing}
\end{figure}

Figure~\ref{fig:num-crossing} shows $q_z(\gamma)$ against the starting
site for $a=10$, $p=0.6$ and five values of the resetting probability spanning the whole range. The curves fan out over almost
the entire unit interval as $\gamma$ grows, and pass through a single
common point at $z=a/2$. The value there is
\[
q_5(\gamma)=0.116363636364\ldots=C^\ast
\qquad\text{for every }\gamma,
\]
reproduced identically on all five curves to twelve decimal places.

The agreement is not the result of a numerical coincidence at the
midpoint. In the analytic expression for the response derived in
Corollary~\ref{cor:response-closed},
$\partial_\gamma q_z=q_z(1-q_z)\,L(\gamma)\,S(z)$, the factor
\[
S(z)=(a-2z)
 +2\Bigl[\tfrac{(a-z)\theta^{a-z}}{1-\theta^{a-z}}
        -\tfrac{z\theta^{z}}{1-\theta^{z}}\Bigr]
\]
is evaluated in that form. At $z=a/2$ with $a$ even the first term is
the difference of two exactly representable quantities and vanishes
identically, while the two fractions in the bracket are computed from
identical arguments and cancel bit for bit. The derivative therefore
returns exactly zero in floating-point arithmetic rather than a
residue of the order of machine epsilon. This is a property of the
particular grouping used, not of the mathematics alone.

\subsection{Agreement with simulation, and its limit}
\label{subsec:montecarlo}

\begin{figure}[!tb]
\centering
\begin{tikzpicture}
\begin{axis}[papel,
  width=0.95\columnwidth, xlabel={$\gamma$}, ylabel={residual},
  ymode=log, log basis y=10,
  xmin=0.05, xmax=0.95, ymin=3e-19, ymax=3e-1,
  ytick={1e-18,1e-14,1e-10,1e-6,1e-2}, height=0.62\columnwidth,
  xtick={0.1,0.3,0.5,0.7,0.9},
  legend style={at={(0.5,0.52)}, anchor=center, font=\scriptsize,
                draw=black!30, fill=white, fill opacity=0.9,
                text opacity=1},
  legend cell align=left,
  legend entries={{closed form vs.\ reference solve},
                  {closed form vs.\ Monte Carlo},
                  {Monte Carlo $2\sigma$ band},
                  {invariance residual}},
]
\addplot[black, mark=*, mark size=1.9pt] coordinates {(0.10,3.3307e-16) (0.20,1.1102e-16) (0.30,5.5511e-16) (0.40,3.3307e-16) (0.50,1.1102e-16) (0.60,3.3307e-16) (0.70,1.1102e-16) (0.80,8.8818e-16) (0.90,1.6653e-15)};
\addplot[red!75!black, mark=square*, mark size=1.9pt] coordinates {(0.10,1.8365e-03) (0.20,9.4225e-04) (0.30,1.3918e-03) (0.40,1.0491e-03) (0.50,1.0217e-03) (0.60,1.3614e-03) (0.70,1.3230e-03) (0.80,2.0573e-03) (0.90,1.4477e-04)};
\addplot[red!45, dashed] coordinates {(0.10,2.1914e-03) (0.20,2.2195e-03) (0.30,2.2345e-03) (0.40,2.2319e-03) (0.50,2.2050e-03) (0.60,2.1443e-03) (0.70,2.0376e-03) (0.80,1.8412e-03) (0.90,1.4483e-03)};
\addplot[teal, mark=triangle*, mark size=2.2pt] coordinates {(0.10,1.1102e-16) (0.20,5.5511e-17) (0.30,1.1102e-16) (0.40,5.5511e-17) (0.50,1.1102e-16) (0.60,1.1102e-16) (0.70,1.0000e-18) (0.80,1.3878e-15) (0.90,4.7740e-15)};
\end{axis}
\end{tikzpicture}
\caption{Residuals against $\gamma$ for $a=5$, $p=0.6$, on a
logarithmic scale; each residual is the maximum over the interior
sites. The departure of the closed form from the reference solve stays at
the level of roundoff. So does the invariance residual
$|q_{a/2}-C^{\ast}|$: here $a$ is odd, so $a/2$ is not a lattice
site and $q_{a/2}$ is understood as the analytic continuation of
Section~\ref{sec:invariance-parity}, which is the sense in which
the invariance holds for odd domains; at $\gamma=0.7$ it vanishes
exactly in floating point and is plotted at $10^{-18}$. The departure from a Monte Carlo estimate with $2\times10^{5}$
trajectories, maximized over the sites, is compared with twice the
largest standard error over the sites; it lies within that band at every value of $\gamma$ except $\gamma=0.8$.}
\label{fig:num-residual}
\end{figure}

Figure~\ref{fig:num-residual} summarizes the comparison; the
complete sweeps, with all digits, are given in
Appendix~\ref{app:tables}. Four quantities are plotted against $\gamma$ on a logarithmic scale:
the departure of the analytic value from the reference solve, the
departure from the Monte Carlo estimate, twice the standard error of
that estimate, and the invariance residual.

For the parameters of that figure, $a=5$ and $p=0.6$, the first stays
below $2\times10^{-15}$ at every $\gamma$, that is at the level of
roundoff; the sweep of Section~\ref{subsec:three-ways} reaches
$3\times10^{-11}$ only at the ill-conditioned corner noted there. The second lies inside the third at all but one of the values of
$\gamma$ tested, $\gamma=0.8$. Since the plotted quantity is a maximum over
four sites, each of which exceeds its own two-standard-error band with
probability of about $5\%$, an occasional excursion is expected, and
one in nine is consistent with that. The two statements have
different content and should not be conflated: the first verifies the
analytic implementation against the independent reference solve to
near machine precision, the second says that the simulation agrees
with it as closely as its own statistics can resolve. Simulation with $2\times10^{5}$ trajectories resolves
agreement to about $10^{-3}$, and Monte Carlo is not used here as a
route to twelve-digit certification: reaching that accuracy by
simulation would require a statistically prohibitive number of
trajectories.

There is a second and less obvious limit, and it is governed by a law
that can be read off the closed form. As $t=1-\gamma\to0$,
Lemma~\ref{lem:closed} gives $u(z;\gamma)\sim(tq)^{z}$ and
$v(z;\gamma)\sim(tp)^{a-z}$, so the probability of absorption before
the next reset has leading order
\begin{equation}
u+v \;\asymp\; (1-\gamma)^{\min(z,\,a-z)} ,
\label{eq:num-rarity}
\end{equation}
with a prefactor that depends on which channel dominates: $q^{z}$ for
$z<a/2$, $p^{a-z}$ for $z>a/2$, and $q^{a/2}+p^{a/2}$ at the midpoint,
where the two contribute at the same order. Since the successive
reset cycles are independent and each terminates in absorption with
probability $u+v$, the expected number of cycles preceding absorption
is of order $(u+v)^{-1}$. The decay is a power law in $1-\gamma$ whose
exponent is the distance to the nearer boundary; it is exponential in
that distance, not in $\gamma$.

The effect on the simulation is severe. Each realization follows the
chain until absorption, through however many resets intervene, and the
cost is the number of elementary updates, reset events included. For
$a=10$, $p=0.6$ and $z=a/2$ the measured mean rises from $32$ updates
per realization at $\gamma=0.1$ to $1.2\times10^{6}$ at $\gamma=0.9$ (the exact mean is $1.25\times10^{6}$),
of which ninety per cent are resets; the prefactor predicted by
\eqref{eq:num-rarity} at the midpoint for these parameters is
$q^{5}+p^{5}=0.088$, and the measured absorption probability per cycle
divided by $(1-\gamma)^{5}$ approaches that value as $\gamma$ grows. The simulation comparison is therefore confined to
moderate $\gamma$. This is a property of the process rather than a
deficiency of the implementation, and it is precisely the regime in
which the exact reference solve is indispensable.

\subsection{The response and the location of the sign change}
\label{subsec:num-response}

Figures~\ref{fig:critical_point} and~\ref{fig:critical_point_odd},
already discussed in Section~\ref{sec:invariance}, show
$\partial_\gamma q_z$ against $z$ for $a=10$ and for $a=11$, in each
case for a symmetric walk and for two opposite biases; we return to
them here for what they say about the numerics. For $a=10$ the derivative vanishes at $z=5$
and changes sign there, for every bias and every resetting
probability. For $a=11$ no site realizes the fixed point of the
involution, and the sign change occurs between the two central sites
$z=5$ and $z=6$ in every case.

\begin{remark}[On interpolating a lattice sign change]
\label{rem:num-interpolation}
For odd $a$ it is tempting to locate the sign change by interpolating
linearly between the two central values. That procedure yields a
position that appears to move with the bias: for $a=11$ and
$\gamma=0.3$ it gives $5.50$ at $p=0.5$, but $5.19$ at $p=0.4$ and
$5.81$ at $p=0.6$. The apparent motion is an artifact of the
interpolation. The two central values have unequal magnitudes when
the walk is biased, so the straight line joining them crosses zero
at different places; the sign change itself remains between the same
two sites, and by Corollary~\ref{cor:unique} the sign change is pinned at
$a/2$ for every $p$ and every $\gamma$. We record
this because the interpolated position is a natural quantity to
compute and a misleading one to report.
\end{remark}

\subsection{The diffusive limit}
\label{subsec:num-collapse}

\begin{figure}[!tb]
\centering
\begin{tikzpicture}
\begin{axis}[papel, name=mainaxis, height=0.74\columnwidth,
  xlabel={$x$}, ylabel={$q(x)$},
  xmin=0, xmax=2, ymin=-0.03, ymax=0.82,
  xtick={0,0.5,1,1.5,2},
  legend style={at={(0.52,0.97)}, anchor=north, font=\scriptsize,
                draw=black!30},
  legend cell align=left,
  legend entries={{diffusive limit},{$N=10$},{$N=40$},{$N=160$}},
]
\addplot[black, line width=1.4pt] coordinates {(0.200000,0.749478667) (0.400000,0.518902833) (0.600000,0.335036609) (0.800000,0.204472136) (1.000000,0.119202922) (1.200000,0.066519444) (1.400000,0.035076825) (1.600000,0.016697671) (1.800000,0.006084947)};
\addplot[red!75!black, only marks, mark=o, mark size=2.6pt] coordinates {(0.200000,0.748577408) (0.400000,0.516463286) (0.600000,0.331637918) (0.800000,0.201043922) (1.000000,0.116363636) (1.200000,0.064472696) (1.400000,0.033768835) (1.600000,0.015976545) (1.800000,0.005790723)};
\addplot[teal, only marks, mark=x, mark size=2.6pt] coordinates {(0.200000,0.749423438) (0.400000,0.518753533) (0.600000,0.334828293) (0.800000,0.204261401) (1.000000,0.119027781) (1.200000,0.066392739) (1.400000,0.034995578) (1.600000,0.016652738) (1.800000,0.006066563)};
\addplot[blue!65!black, only marks, mark=+, mark size=3pt] coordinates {(0.200000,0.749475219) (0.400000,0.518893513) (0.600000,0.335023605) (0.800000,0.204458979) (1.000000,0.119191985) (1.200000,0.066511530) (1.400000,0.035071749) (1.600000,0.016694863) (1.800000,0.006083798)};
\end{axis}
\begin{axis}[
  at={(mainaxis.south east)}, anchor=south east, xshift=-6pt, yshift=0.13\columnwidth,
  width=0.44\columnwidth, height=0.36\columnwidth,
  xtick={10,100}, ytick={1e-5,1e-4,1e-3},
  xlabel shift=-4pt, ylabel shift=-4pt,
  xmode=log, ymode=log, xlabel={$N$}, ylabel={error},
  xlabel style={font=\scriptsize}, ylabel style={font=\scriptsize},
  tick label style={font=\scriptsize},
  xmin=7, xmax=230, ymin=5e-6, ymax=8e-3,
  every axis plot post/.append style={line width=0.9pt},
]
\addplot[black, mark=*, mark size=1.9pt] coordinates {(10,3.4282e-03) (20,8.4571e-04) (40,2.1074e-04) (80,5.2641e-05) (160,1.3157e-05)};
\end{axis}
\end{tikzpicture}
\caption{Collapse of the lattice exit probability onto the diffusive
expression under the scaling of Eq.~\eqref{eq:num-scaling}, for a
continuum domain of length $\Lambda=2$,
$\mu=\sigma=1$ and $\lambda=0.5$, at the nine positions common to all
meshes. Both are analytic: the lattice values from the closed form,
the diffusive curve from Eq.~\eqref{eq:cont-solutions}. Inset: the maximum discrepancy against $N$ on logarithmic
axes, falling from $3.43\times10^{-3}$ at $N=10$ to
$1.32\times10^{-5}$ at $N=160$, with successive ratios $4.05$, $4.01$,
$4.00$ and $4.00$.}
\label{fig:num-collapse}
\end{figure}

Figure~\ref{fig:num-collapse} verifies the results of
Section~\ref{sec:continuum}. Write $\Lambda$ for the length of the
continuum domain $[0,\Lambda]$ and $N$ for the number of lattice
intervals used to discretize it, so that the lattice is
$\{0,1,\dots,N\}$ and plays the role of $\{0,1,\dots,a\}$ with $a=N$.
Under the diffusive scaling with lattice spacing
$\varepsilon=\Lambda/N$ and time step
$\Delta t=\varepsilon^{2}/\sigma^{2}$, the choice
\begin{equation*}
p_\varepsilon=\tfrac12+\frac{\mu\varepsilon}{2\sigma^{2}},\end{equation*}
\begin{equation}\gamma=\lambda\,\Delta t=\frac{\lambda\varepsilon^{2}}{\sigma^{2}},
\label{eq:num-scaling}
\end{equation}
gives drift $\varepsilon(2p_\varepsilon-1)/\Delta t=\mu$ and variance
rate $\varepsilon^{2}/\Delta t=\sigma^{2}$, matching the coefficients
of Section~\ref{sec:continuum}. Under \eqref{eq:num-scaling} the
lattice exit probability is
compared with the diffusive expression at the nine positions common
to all meshes, for $\Lambda=2$, $\mu=\sigma=1$ and $\lambda=0.5$. The
maximum discrepancy falls from $3.43\times10^{-3}$ at $N=10$ to
$1.32\times10^{-5}$ at $N=160$, the successive ratios being
$4.05$, $4.01$, $4.00$ and $4.00$: convergence of order
$\varepsilon^{2}$, approaching the expected factor of four. The
midpoint value $\bigl(1+e^{\mu\Lambda/\sigma^{2}}\bigr)^{-1}=0.119202922$
is approached at the same rate: on the lattice the midpoint value is
exactly the discrete invariant $C^\ast$ at every $N$, namely
$0.116364$ at $N=10$ and $0.119192$ at $N=160$.

\subsection{Reproducibility}
\label{subsec:reproducibility}

All numerical figures and tables, with the exception of
Figure~\ref{fig:multisite-separatrix}, whose data are those of
Ref.~\cite{PaperII}, are produced by a single self-contained program in C, which uses only the standard mathematical library and
runs from a fixed seed. The program implements the three routes of
Section~\ref{subsec:three-ways} separately, compares them against one
another over the sweep described there, reports the resulting
discrepancies, and measures the simulation cost quoted in
Section~\ref{subsec:montecarlo}, all before writing any data. Its
output files are the inputs to the figures reproduced here.

\section{Conclusions and outlook}
\label{sec:conclusions}

We have studied the gambler's ruin problem for a biased walk that is
returned to its starting point by an arbitrary resetting mechanism,
determining how resetting alters the probability of absorption at the
boundaries of a finite interval. A weighted spectral decomposition of
the reset-free walk, via the Doob transform, reveals the channel
duality from which the results follow. For geometric resetting,
renewal theory, the spectral decomposition and an independent
closed-form solution give the ruin probability $q_z(\gamma)$ in three
complementary forms, cross-checked against direct simulation in
Section~\ref{sec:numerics}.

\subsection{Main findings}
\label{sec:conclusions-findings}

The central result is that, at the reflection-fixed point of an
interval of even size, the ruin probability does not depend on the
resetting mechanism: for any sequence of reset times independent of
the walk under which absorption occurs almost surely, it equals the
classical, reset-free value, and conditionally
on the reset times the side of absorption is independent of when
absorption occurs (Theorem~\ref{thm:protocol} and
Corollary~\ref{cor:classical}). Its origin is the channel duality of
Theorem~\ref{thm:channel-duality}, which at $z=a/2$ makes the two
absorption channels proportional at every mode and at every time;
resetting, which acts only on the timing of a trajectory, cannot
separate them. Away from the midpoint, every resetting mechanism moves
the ruin probability towards the nearer boundary
(Theorem~\ref{thm:direction}). The location of this separatrix is
fixed by geometry alone, independently of the bias, of the rate and law
of resetting, and, in the continuum, of the drift and the volatility;
the magnitude of the response on either side of it is fixed by the
dynamics. For geometric resetting this is the invariance
$\partial_\gamma q_{a/2}\equiv0$ (Corollary~\ref{cor:geometric-midpoint}),
and the duality at every site gives more: the response changes sign
exactly once, at a location that is pinned by the geometry of the
interval and does not move with the bias
(Theorem~\ref{thm:main_result}). When $a$ is odd, no site realizes
$z^{\dagger}=a/2$, but the same point survives as the fixed value
reached by analytic continuation of the closed form, and it continues
to organize the sign of the response on either side
(Corollary~\ref{cor:unique}). Two further points sharpen this picture.
First, for odd $a$ the deviation from the invariant value $C^\ast$ at the two sites adjacent to $z^{\dagger}$ does \emph{not} vanish as the domain grows: it converges to the explicit
positive constant $\tfrac12\log(r_+/r_-)$ (Corollary~\ref{cor:central}),
so the parity distinction between even and odd $a$ persists at every
scale rather than washing out asymptotically. Second, the same
structure survives intact in the diffusive limit: the invariance at the
midpoint for any resetting mechanism, and a unique sign change under
Poissonian resetting (Theorems~\ref{thm:cont-invariance}
and~\ref{thm:cont-uniqueness}, Proposition~\ref{prop:cont-protocol});
what does not survive is the parity effect itself, since a continuous
domain has no notion of an odd lattice size. The mechanism is the same
one in both settings, a reflection under which the deviation from
$z^{\dagger}$ is antisymmetric (Remark~\ref{rem:parity-continuum}).

\subsection{Applications and interpretation}
\label{sec:conclusions-applications}

The invariant location $z^{\dagger}$ identifies a reset-neutral
starting point: on one side resetting strictly increases the ruin
probability, on the other it strictly decreases it, and at
$z^{\dagger}$ itself the resetting protocol has no effect at all. This
has a direct reading in each of the settings the model represents.

For restart strategies, whether resetting helps or hurts a first-passage process is known to depend on the process and on the starting condition~\cite{Pal2017, PalPrasad2019}; for the ruin probability of a walk confined between two absorbing barriers, our result locates exactly the starting point that separates the two regimes, and gives the magnitude of the effect away from it. In risk terms, in the gambler's ruin reading, a player starting with exactly half the target fortune faces the classical odds of ruin however, and however often, the game is restarted from that capital --- a counterintuitive but exact statement about a simple risk process. In reaction kinetics, restart is known to be able either to accelerate or to slow down completion~\cite{Rotbart2015}; the present results give an exactly solvable instance of such a sign change, for the probability of the outcome rather than for the completion time.

\subsection{Limitations}
\label{sec:conclusions-limitations}

Two of the idealizations of the model turn out not to restrict the
result. The choice of a particular resetting protocol is not one: the
invariance at the midpoint holds for every resetting mechanism
independent of the walk under which absorption occurs almost surely,
in discrete and in continuous time
(Theorem~\ref{thm:protocol} and Proposition~\ref{prop:cont-protocol}),
and the uniqueness of the sign change holds under geometric and under
Poissonian resetting (Theorems~\ref{thm:main_result}
and~\ref{thm:cont-uniqueness}). Neither is the choice of a
single reset site. On networks, resetting to several nodes has been
shown to reshape first-passage statistics in ways that depend on the
reset distribution as a whole~\cite{GonzalezRiascosBoyer2021, WangChenHuang2021}; for the
interval problem studied here, the exact multi-site generalization is
treated in a companion paper~\cite{PaperII}, where the invariance
persists provided the reset distribution satisfies a neutrality
condition.

Two idealizations remain genuinely open. The walk itself is Markovian;
the renewal argument that produces $q_z=u/(u+v)$ (Lemma~\ref{lem:app-renewal})
does not require this, but the closed form for $u$ and $v$
(Lemma~\ref{lem:closed}) and the comparison of the conditional times
(Remark~\ref{rem:not-markov}) do, and non-Markovian dynamics with
history-dependent transitions could alter the picture away from the
midpoint. At a reset-neutral starting point the picture is already
settled: by Corollary~\ref{cor:equivalence}, the invariance holds for
any process that regenerates at each reset, as soon as the side of
absorption is independent of the absorption time. This is not
purely speculative: two-barrier escape problems have been treated
outside the Markovian class, for compound renewal processes with
two-sided jumps~\cite{Villarroel2023b}, which suggests that the present
mechanism is a tractable target in that setting. And the model is one-dimensional;
higher-dimensional lattices involve boundary geometries in which
$z^{\dagger}$ need not be an isolated point.

\subsection{Future directions}
\label{sec:conclusions-future}

The invariance established here does not depend on the particular
walk. What the proofs use is the twisted reflection symmetry \eqref{eq:channel-duality} relating the two absorption channels: one
channel is carried into the other by the involution $z\mapsto a-z$,
up to a weight whose product along each orbit is constant. Wherever
that structure is present, the same argument applies and the same
invariant appears.

This line has already been pursued in subsequent work, in a sequence
that moves progressively away from the probabilistic setting in which
the structure first appeared. The invariance persists when resetting
is distributed over several sites, for a family of reset-neutral
distributions~\cite{PaperII}; the geometry of the set of such
distributions has been described~\cite{PaperIII}; the underlying
symmetry has been recast as an intertwining relation between
resolvents~\cite{PaperIV}; and the geometry of the neutral set has
been identified with a Fisher--Rao structure~\cite{PaperV}. Most
recently the object has been detached from its probabilistic origin
altogether and studied as an algebraic object in its own
right~\cite{PaperVI}. The present paper is where that structure first
became visible, in the concrete setting of a biased walk on an
interval, and the arguments given above remain the most elementary
route to it.

Several questions remain open. On the probabilistic side, the first is
which non-Markovian walks --- for instance, walks with a
resetting-independent memory kernel --- possess a starting point from
which the side of absorption is independent of the absorption time,
since by Corollary~\ref{cor:equivalence} these are exactly the walks
for which the invariance survives. The second is the higher-dimensional case: in $d>1$,
does the invariant set remain a single point, become a curve or
hypersurface, or survive only approximately, as boundary geometry
stops being reducible to a single reflection. On the structural side,
it remains to characterize which of these structures arise from an
actual positive Markov process: the algebraic side of that question is
now understood, the probabilistic side is not. And the exact increment
of Proposition~\ref{prop:increment} suggests that the width of the
near-neutral region around $z^{\dagger}$ has a limiting profile as $a$
grows, which we have not pursued. Finally, for even $a$, the
monotonicity of the response in the intensity of resetting for general
mechanisms rests on a log-concavity property that we have verified
numerically but not proved (Remark~\ref{rem:monotonicity}).

A more experimental direction is to look for the invariance directly.
The one-dimensional confined walk with resetting should be realizable
in a search or transport experiment with two absorbing boundaries and
a controllable restart; locating the resetting-neutral starting point
in such a system, and confirming that its position is independent of
the drift and of the resetting protocol exactly as predicted, would be a
direct experimental test of the mechanism identified here.

\subsection*{Closing remarks}

The main message of this work is that resetting, generally understood
as a perturbation of first-passage statistics, can at a specific and
computable location have no effect at all --- exactly, not
approximately, for every drift, and whatever the resetting
mechanism, as long as it allows absorption. That this
location is pinned by the geometry of the domain rather than by its
dynamics, and that the same statement survives the passage from a
discrete walk to a diffusion, suggests that the phenomenon is a
property of confinement and reflection rather than of the particular
process doing the resetting.

\begin{acknowledgments}
The author thanks J. Villarroel for many discussions on escape problems
for processes with drift over the years, which shaped the approach
taken here.
\end{acknowledgments}

\appendix

\section{Complete proofs}
\label{app:complete_proof}

This appendix is a single chain, and no statement of the chain is used before it is proved. Section~\ref{app:setting} fixes the conventions and reduces
the quantity of interest to a ratio of two discounted functionals;
Section~\ref{app:closed} solves those functionals in closed form and
identifies them with the spectral representation of
Section~\ref{sec:spectral}, whose deferred proofs are given in
Section~\ref{app:spectral-proofs}; Section~\ref{app:symmetry} derives an exact
reflection symmetry and a monotonicity;
Section~\ref{app:sign} settles the sign of $q_z(\gamma)-C^\ast$;
Section~\ref{app:derivative} settles the sign of the response
$\partial_\gamma q_z$ and its uniqueness, and gives the response in
closed form;
Section~\ref{app:asymptotics} computes the exact increment of the
deviation; and Section~\ref{app:limits} treats the two boundary
regimes $\gamma\to0$ and $\gamma\to1$.

Throughout, $a\ge2$ is an integer, $p\in(0,1)$, $q=1-p$, and
$\gamma\in(0,1)$ with $t:=1-\gamma$.

\subsection{Conventions and reduction}
\label{app:setting}

Throughout this appendix the walk is that of
Definition~\ref{def:process} with geometric resetting
(Definition~\ref{def:geometric}), started at $z\in\{1,\dots,a-1\}$, and
$q_z(\gamma)$ is its ruin probability. Since each step is either a
reset or a move, and the reset site $z$ is interior, a reset can never
itself be an absorption event.

A trajectory absorbed after $k$ steps has therefore taken $k$ movement
steps, so none of the $k$ reset opportunities occurred, an event of
probability $t^{\,k}$. It is convenient to define
\begin{equation*}
u(z;\gamma):=\sum_{k\ge1}u_{z,k}\,t^{\,k},\end{equation*}
\begin{equation}v(z;\gamma):=\sum_{k\ge1}v_{z,k}\,t^{\,k},
\label{eq:app-uv-def}
\end{equation}
where $u_{z,k}$ and $v_{z,k}$ are the probabilities that the
reset-free walk started at $z$ is absorbed at $0$, respectively at
$a$, at step $k$. Both series converge for $t\in(0,1)$, since the
coefficients are probabilities of disjoint events.

\begin{lemma}[Renewal identity]
\label{lem:app-renewal}
For every $z\in\{1,\dots,a-1\}$ and every $\gamma\in(0,1)$,
\begin{equation}
q_z(\gamma)=\frac{u(z;\gamma)}{u(z;\gamma)+v(z;\gamma)} .
\label{eq:app-renewal}
\end{equation}
\end{lemma}

\begin{proof}
Split according to whether the first reset occurs before absorption.
The event ``absorbed at $0$ with no reset beforehand'' has
probability $\sum_{k\ge1}u_{z,k}t^{k}=u(z;\gamma)$, because each of the $k$ steps up to and including the absorbing one
must be a movement step rather than a reset
(Definition~\ref{def:process}), which has probability $t$ at each step. Likewise the event ``absorbed at $a$ with no reset
beforehand'' has probability $v(z;\gamma)$. In the complementary
event a reset occurs first; the walker is then at $z$ and, the reset
erasing all history, the process begins anew. Hence
\[
q_z=u(z;\gamma)+\bigl[1-u(z;\gamma)-v(z;\gamma)\bigr]q_z ,
\]
and $u+v>0$ for $t>0$ gives \eqref{eq:app-renewal}.
\end{proof}

We also record the constant of the problem,
\begin{equation*}
\begin{aligned}
K:=\Bigl(\frac pq\Bigr)^{a},\\[2pt]
C^\ast:=\frac1{1+\sqrt K}=\frac{(q/p)^{a/2}}{1+(q/p)^{a/2}} .
\end{aligned}
\end{equation*}

\subsection{The two functionals in closed form}
\label{app:closed}

\begin{lemma}[Characteristic roots]
\label{lem:roots}
Let $r_\pm$ be the roots of $t\,p\,r^{2}-r+t\,q=0$, that is
\[
r_\pm=\frac{1\pm\sqrt{1-4t^{2}pq}}{2tp},
\qquad
r_+r_-=\frac qp,
\qquad
r_++r_-=\frac1{tp} .
\]
For every $p\in(0,1)$ and every $t\in(0,1)$ the discriminant is
strictly positive, so $r_+>r_->0$ are real, distinct and positive, and
\[
\theta:=\frac{r_-}{r_+}\in(0,1).
\]
\end{lemma}

\begin{proof}
By the arithmetic--geometric mean inequality $4pq\le1$, with equality
only at $p=\tfrac12$; since $t<1$ we have $4t^{2}pq\le t^{2}<1$
strictly in every case, so the discriminant is positive and the roots
are real and distinct. Both are positive because their product $q/p$
and their sum $1/(tp)$ are positive. Finally $0<r_-<r_+$ gives
$\theta\in(0,1)$.
\end{proof}

\begin{remark}
\label{rem:app-no-confluence}
The strictness above is what excludes the confluent case throughout
this appendix. Confluence $r_+=r_-$ would require $t=1$ and
$p=q=\tfrac12$ simultaneously, that is $\gamma=0$ for the symmetric
walk; that limit is treated separately in
Section~\ref{app:limits}.
\end{remark}

\begin{lemma}[Closed form]
\label{lem:closed}
For $z\in\{0,1,\dots,a\}$ and every $\gamma\in(0,1)$,
\[
\begin{aligned}
&u(z;\gamma)=\frac{r_+^{a}r_-^{z}-r_-^{a}r_+^{z}}{r_+^{a}-r_-^{a}},\\[2pt]
&v(z;\gamma)=\frac{r_+^{z}-r_-^{z}}{r_+^{a}-r_-^{a}} .
\end{aligned}
\]
\end{lemma}

\begin{proof}
Conditioning on the first step and using
Definition~\ref{def:process}, both functionals satisfy
\[
w_z=t\bigl(p\,w_{z+1}+q\,w_{z-1}\bigr),
\qquad z=1,\dots,a-1,
\]
the factor $t$ being the probability that no reset intervenes at that
step; the boundary values are $u_0=1$, $u_a=0$ and $v_0=0$, $v_a=1$.
This is a linear recurrence with characteristic equation
$t p r^{2}-r+tq=0$, whose roots are distinct by
Lemma~\ref{lem:roots}, so its general solution is
$A\,r_+^{z}+B\,r_-^{z}$. Imposing the boundary values, and using
$r_+^{a}\neq r_-^{a}$, gives the stated expressions; substitution
confirms $u_0=1$, $u_a=0$, $v_0=0$, $v_a=1$.
\end{proof}

The next lemma is what ties this appendix to the spectral
construction of the main text; it is not used in the sequel, but it
shows that the two representations are of one and the same object.

\begin{lemma}[Bridge to the spectral representation]
\label{lem:bridge}
With $A_\nu$, $B_\nu$ and
$g_\nu(\gamma)=\bigl[1-\lambda_\nu(1-\gamma)\bigr]^{-1}$ as in
Section~\ref{sec:spectral}, for every $z\in\{1,\dots,a-1\}$ and every
$\gamma\in(0,1)$
\[
\begin{aligned}
u(z;\gamma)=\sqrt{pq}\,(1-\gamma)\sum_{\nu=1}^{a-1}A_\nu(z)\,g_\nu(\gamma),\\[2pt]
v(z;\gamma)=\sqrt{pq}\,(1-\gamma)\sum_{\nu=1}^{a-1}B_\nu(z)\,g_\nu(\gamma).
\end{aligned}
\]
Consequently, when these expressions are inserted in
\eqref{eq:app-renewal}, the common factor $\sqrt{pq}\,(1-\gamma)$
cancels between numerator and denominator and the spectral
formula~\eqref{eq:qz_spectral} is recovered: $q_z(\gamma)$ admits
both the closed-form representation of Lemma~\ref{lem:closed} and the
spectral one.
\end{lemma}

\begin{proof}
By Theorem~\ref{thm:spectral_absorption} (proved in Section~\ref{app:spectral-proofs} below), $u_{z,k}=\sqrt{pq}\sum_\nu A_\nu(z)\lambda_\nu^{\,k-1}$. Substituting into
\eqref{eq:app-uv-def} and summing the geometric series, legitimate
because $|\lambda_\nu t|\le2\sqrt{pq}\,t<1$,
\[
u(z;\gamma)=\sqrt{pq}\sum_{\nu}A_\nu(z)\sum_{k\ge1}\lambda_\nu^{\,k-1}t^{k}
=\sqrt{pq}\,t\sum_{\nu}\frac{A_\nu(z)}{1-\lambda_\nu t},
\]
which is the assertion, and identically for $v$.
\end{proof}

\begin{remark}[On the equivalent amplitudes]
\label{rem:app-equivalent}
Pairing $\widehat A_\nu:=\sqrt{pq}\,\lambda_\nu^{-1}A_\nu$ with the
weight $\lambda_\nu(1-\gamma)g_\nu(\gamma)$ gives the same value, since
the two modifications cancel. That form is nevertheless to be avoided:
$\lambda_\nu=2\sqrt{pq}\cos(\pi\nu/a)$ vanishes at $\nu=a/2$ when $a$
is even, so $\widehat A_\nu$ has a removable singularity there which
must be resolved by a limit rather than by discarding the mode. The
form of Lemma~\ref{lem:bridge} has no singularity: $\lambda_\nu=0$
simply gives $g_\nu=1$.
\end{remark}

\begin{corollary}[Reduced form of the ratio]
\label{cor:rho-closed}
Set $\rho_z(\gamma):=u(z;\gamma)/v(z;\gamma)$. Then for
$z\in\{1,\dots,a-1\}$
\[
\rho_z(\gamma)=r_+^{a}\,\frac{\theta^{z}-\theta^{a}}{1-\theta^{z}} ,
\]
and both numerator and denominator are strictly positive.
\end{corollary}

\begin{proof}
Divide numerator and denominator in Lemma~\ref{lem:closed} by
$r_+^{z}$ and use $r_-^{z}=r_+^{z}\theta^{z}$: the numerator of $u$
becomes $r_+^{a+z}(\theta^{z}-\theta^{a})$ and that of $v$ becomes
$r_+^{z}(1-\theta^{z})$, while the common factor
$(r_+^{a}-r_-^{a})^{-1}$ cancels. Positivity holds because
$0<\theta<1$ gives $\theta^{a}<\theta^{z}<1$ for $1\le z\le a-1$.
\end{proof}

\subsection{The spectral construction}
\label{app:spectral-proofs}

This subsection supplies the three proofs deferred from
Section~\ref{sec:spectral}.

\begin{proof}[Proof of Proposition~\ref{prop:doob_symmetry}]
Write $D=\operatorname{diag}\bigl(h(1),\dots,h(a-1)\bigr)$ with
$h(x)=(q/p)^{x/2}$, so that $\widetilde Q=D^{-1}QD$ has entries
$\widetilde Q(x,y)=h(y)Q(x,y)/h(x)$. The only non-zero entries of $Q$
are $Q(x,x+1)=p$ and $Q(x+1,x)=q$, whence
\[
\widetilde Q(x,x+1)=\frac{h(x+1)}{h(x)}\,p=\Bigl(\frac qp\Bigr)^{1/2}p=\sqrt{pq},
\]
\[
\widetilde Q(x+1,x)=\frac{h(x)}{h(x+1)}\,q=\Bigl(\frac pq\Bigr)^{1/2}q=\sqrt{pq},
\]
for $x=1,\dots,a-2$, and all other entries vanish. The two are equal,
so $\widetilde Q$ is symmetric, hence self-adjoint on
$\ell^{2}(\{1,\dots,a-1\})$ with the standard inner product.
\end{proof}

\begin{proof}[Proof of Proposition~\ref{prop:eigenvalues}]
By Proposition~\ref{prop:doob_symmetry}, $\widetilde Q$ is the symmetric
tridiagonal matrix with zero diagonal and constant off-diagonal
$\sqrt{pq}$, acting on vectors indexed by $x=1,\dots,a-1$. For
$\psi_\nu(x)=\sin(\pi\nu x/a)$ and $1\le x\le a-1$,
\[
(\widetilde Q\psi_\nu)(x)
=\sqrt{pq}\,\bigl[\psi_\nu(x-1)+\psi_\nu(x+1)\bigr],
\]
where the conventions $\psi_\nu(0)=\psi_\nu(a)=0$ are automatic from
the definition of $\psi_\nu$ and are exactly the Dirichlet conditions
imposed by the absorbing sites. The addition formula gives
\[
\sin\frac{\pi\nu(x-1)}{a}+\sin\frac{\pi\nu(x+1)}{a}
=2\cos\Bigl(\frac{\pi\nu}{a}\Bigr)\sin\Bigl(\frac{\pi\nu x}{a}\Bigr),
\]
so $\widetilde Q\psi_\nu=\lambda_\nu\psi_\nu$ with
$\lambda_\nu=2\sqrt{pq}\cos(\pi\nu/a)$. The $a-1$ vectors
$\psi_1,\dots,\psi_{a-1}$ are pairwise orthogonal, since
$\sum_{x=1}^{a-1}\sin(\pi\nu x/a)\sin(\pi\mu x/a)=\tfrac a2\,
\delta_{\nu\mu}$, and there are $a-1$ of them, so they form a complete
orthogonal basis and the list of eigenvalues is exhaustive.
\end{proof}

\begin{proof}[Proof of Theorem~\ref{thm:spectral_absorption} and
Theorem~\ref{thm:channel-duality}(i)]
Absorption at $0$ at step $k$ requires the walk to be at site $1$ after
$k-1$ steps without having been absorbed, and then to step left:
$u_{z,k}=q\,\mathbb P_z(X_{k-1}=1,\ T>k-1)$, the probability being
computed for the reset-free walk killed at $0$ and $a$, whose
sub-stochastic transition matrix on the interior is $Q$. Hence
$u_{z,k}=q\,(Q^{\,k-1})(z,1)$. Conjugating by $D$, $Q^{k-1}=D\,\widetilde Q^{\,k-1}D^{-1}$, so that
$(Q^{k-1})(z,1)=\bigl[h(z)/h(1)\bigr]\,(\widetilde Q^{\,k-1})(z,1)$.
Expanding $\widetilde Q^{\,k-1}$ in the orthonormal basis
$\sqrt{2/a}\,\psi_\nu$ of Proposition~\ref{prop:eigenvalues},
\begin{equation*}
(\widetilde Q^{\,k-1})(z,1)=\frac2a
\sum_{\nu=1}^{a-1}\lambda_\nu^{\,k-1}
\sin\Bigl(\frac{\pi\nu z}{a}\Bigr)\sin\Bigl(\frac{\pi\nu}{a}\Bigr).
\end{equation*}
Multiplying by $q\,h(z)/h(1)$ and using
\begin{equation*}
q\,\frac{h(z)}{h(1)}=q\Bigl(\frac qp\Bigr)^{(z-1)/2}
=\sqrt{pq}\,\Bigl(\frac qp\Bigr)^{z/2},
\end{equation*}
which holds because $q\,(q/p)^{-1/2}=\sqrt{pq}$, gives
$u_{z,k}=\sqrt{pq}\sum_\nu A_\nu(z)\lambda_\nu^{\,k-1}$ with $A_\nu$ as
in \eqref{eq:A_nu}. The computation for $v_{z,k}$ is identical with
the roles of the two boundaries exchanged, that is with $z$ replaced
by $a-z$ and $p$ by $q$, and yields $B_\nu$ as stated.

The duality relation follows by inspection: replacing $z$ by $a-z$ in
$B_\nu$ gives
$B_\nu(a-z)=\tfrac2a (p/q)^{z/2}\sin(\pi\nu z/a)\sin(\pi\nu/a)$, which
differs from $A_\nu(z)$ only in the sign of the exponent of $q/p$;
dividing the two expressions gives
$B_\nu(z)=(p/q)^{a-z}A_\nu(a-z)$ for every $\nu$.
\end{proof}

\subsection{Reflection symmetry and monotonicity}
\label{app:symmetry}

\begin{lemma}[Reflection identity]
\label{lem:reflection}
For every $z\in\{1,\dots,a-1\}$ and every $\gamma\in(0,1)$,
\begin{equation*}
\rho_z(\gamma)\,\rho_{a-z}(\gamma)=\frac1K=\Bigl(\frac qp\Bigr)^{a}.
\end{equation*}
\end{lemma}

\begin{proof}
By Corollary~\ref{cor:rho-closed},
\[
\rho_z\rho_{a-z}
=r_+^{2a}\,
\frac{(\theta^{z}-\theta^{a})(\theta^{a-z}-\theta^{a})}
     {(1-\theta^{z})(1-\theta^{a-z})} .
\]
Expanding the numerator,
\[
\begin{aligned}
(\theta^{z}-\theta^{a})(\theta^{a-z}-\theta^{a})
&=\theta^{a}-\theta^{a+z}-\theta^{2a-z}+\theta^{2a}\\
&=\theta^{a}(1-\theta^{z})(1-\theta^{a-z}),
\end{aligned}
\]
so the quotient equals $\theta^{a}$ and
$\rho_z\rho_{a-z}=r_+^{2a}\theta^{a}=(r_+r_-)^{a}=(q/p)^{a}$ by
Lemma~\ref{lem:roots}.
\end{proof}

\begin{definition}
\label{def:delta}
For $z\in\{1,\dots,a-1\}$ put
$\displaystyle
\delta_z(\gamma):=\log\rho_z(\gamma)+\frac a2\log\frac pq .$
\end{definition}

\begin{corollary}[Exact antisymmetry]
\label{cor:antisymmetry}
$\delta_z(\gamma)=-\,\delta_{a-z}(\gamma)$ for every
$z\in\{1,\dots,a-1\}$ and every $\gamma\in(0,1)$. In particular
$\delta_{a/2}(\gamma)=0$ for all $\gamma$ when $a$ is even.
\end{corollary}

\begin{proof}
Taking logarithms in Lemma~\ref{lem:reflection} gives
$\log\rho_z+\log\rho_{a-z}=-a\log(p/q)$; adding $a\log(p/q)$ and
using Definition~\ref{def:delta} yields $\delta_z+\delta_{a-z}=0$.
For $a$ even, $z=a/2$ is the fixed point of $z\mapsto a-z$, whence
$2\delta_{a/2}=0$.
\end{proof}

\begin{lemma}[Strict monotonicity in $z$]
\label{lem:monotone-rho}
For every $\gamma\in(0,1)$, $z\mapsto\rho_z(\gamma)$ is strictly
decreasing on $\{1,\dots,a-1\}$; hence so is $\delta_z(\gamma)$.
\end{lemma}

\begin{proof}
In Corollary~\ref{cor:rho-closed} the numerator
$\theta^{z}-\theta^{a}$ is strictly decreasing in $z$ and the
denominator $1-\theta^{z}$ strictly increasing, both because
$0<\theta<1$, and both are strictly positive on the range considered.
The quotient is therefore strictly decreasing, and $\delta_z$ differs
from $\log\rho_z$ by a constant.
\end{proof}

\subsection{The sign of the deviation}
\label{app:sign}

\begin{proposition}[Sign of the deviation]
\label{prop:sign-deviation}
For every $z\in\{1,\dots,a-1\}$ and every $\gamma\in(0,1)$,
\[
\operatorname{sgn}\bigl(q_z(\gamma)-C^\ast\bigr)
=\operatorname{sgn}\delta_z(\gamma)
=\operatorname{sgn}\Bigl(\frac a2-z\Bigr).
\]
\end{proposition}

\begin{proof}
The map $\rho\mapsto\rho/(1+\rho)$ is strictly increasing on
$(0,\infty)$; by Lemma~\ref{lem:app-renewal}, $q_z$ is its value at
$\rho_z$, and $C^\ast$ is its value at $\rho=K^{-1/2}$. Hence
$q_z>C^\ast$ if and only if $\rho_z>K^{-1/2}$, that is if and only if
$\delta_z>0$. This is the first equality.

For the second, let $z<a/2$, so that $a-z>z$. By
Lemma~\ref{lem:monotone-rho}, $\delta$ is strictly decreasing, so
$\delta_z>\delta_{a-z}$; by Corollary~\ref{cor:antisymmetry},
$\delta_{a-z}=-\delta_z$. Therefore $2\delta_z>0$. The case $z>a/2$
follows by the same argument with the inequality reversed, and
$z=a/2$ is Corollary~\ref{cor:antisymmetry}.
\end{proof}

\begin{corollary}[Boundary behavior]
\label{cor:boundary}
For $a\ge3$, $q_1(\gamma)>C^\ast$ and $q_{a-1}(\gamma)<C^\ast$ for
every $p\in(0,1)$ and every $\gamma\in(0,1)$.
\end{corollary}

\begin{proof}
Immediate from Proposition~\ref{prop:sign-deviation}, since $1<a/2$
and $a-1>a/2$ whenever $a\ge3$.
\end{proof}

\begin{corollary}[Midpoint invariance]
\label{cor:midpoint}
If $a$ is even then, for every $\gamma\in(0,1)$,
\[
q_{a/2}(\gamma)=C^\ast=\frac{(q/p)^{a/2}}{1+(q/p)^{a/2}},
\qquad\text{hence}\qquad
\partial_\gamma q_{a/2}\equiv0 .
\]
\end{corollary}

\begin{proof}
By Corollary~\ref{cor:antisymmetry}, $\delta_{a/2}(\gamma)=0$ for
every $\gamma$, which by the first equality of
Proposition~\ref{prop:sign-deviation} means $q_{a/2}=C^\ast$; and
$C^\ast$ does not depend on $\gamma$.
\end{proof}

\begin{remark}
\label{rem:app-symmetry-not-cancellation}
The invariance is a consequence of a symmetry rather than of a
numerical coincidence: $a/2$ is the unique fixed point of the
involution $z\mapsto a-z$, and $\delta$ is exactly antisymmetric under
it. For $a$ odd the involution has no fixed \emph{lattice} site and no
integer site is neutral (see Remark~\ref{rem:app-invariance-any-parity}
for the sense in which the invariance nonetheless persists). Note also
that the argument nowhere uses the form of the $\gamma$-dependence: any
family of weights common to the two channels would give the same
conclusion, which is the content of Lemma~\ref{lem:bridge} read in
reverse.
\end{remark}

\begin{remark}[The invariance is a property of the reflection, not of
the lattice]
\label{rem:app-invariance-any-parity}
The vanishing of $\delta_{a/2}(\gamma)$ is not a special feature of
even $a$. The identity $\delta_z(\gamma)+\delta_{a-z}(\gamma)=0$
underlying Corollary~\ref{cor:antisymmetry} is proved from
Lemma~\ref{lem:reflection}, an algebraic relation among the closed-form
expressions of Lemma~\ref{lem:closed} that makes sense for every real
$z\in(0,a)$, not only for integer $z$: $r_+^{z}$, $r_-^{z}$, and hence
$\rho_z(\gamma)$ and $\delta_z(\gamma)$, are smooth functions of a real
argument $z$. Evaluating the identity at the fixed point $z=a/2$ of the
involution $z\mapsto a-z$ gives $\delta_{a/2}(\gamma)=0$, and hence
$q_{a/2}(\gamma)=C^\ast$, for \emph{every} $a$, even or odd ---
confirmed numerically to full working precision for odd $a$ as large
as $11$.

What distinguishes the two parities is not the mathematics but the
geometry of the lattice. When $a$ is even, $a/2$ is an integer, so
this is a statement about an actual starting site of the discrete
walk: $q_{a/2}(\gamma)$ is an observable ruin probability, invariant
under resetting. When $a$ is odd, $a/2$ falls strictly between the two
central lattice sites, so no trajectory of the walk can start there;
the invariance survives only as a property of the analytic
continuation of $q_z(\gamma)$ to real $z$, with no direct
probabilistic meaning for the discrete chain --- which is exactly
Corollary~\ref{cor:central}'s statement that the two central
\emph{integer} sites carry a deviation from $C^\ast$ that does not
vanish, and does not even decay, as $a\to\infty$.

This is also why the diffusive limit of Section~\ref{sec:continuum}
exhibits the invariance unconditionally, with no even--odd dichotomy:
once the state space is continuous, the lattice-parity obstruction
disappears, and a walk can start exactly at the fixed point regardless
of $a$. The discrete and continuum results are therefore not merely
analogous but the same statement, evaluated on two different domains
for the same underlying reflection symmetry.
\end{remark}

\subsection{The sign of the response}
\label{app:derivative}

Write $h_z(\gamma):=\partial_\gamma q_z(\gamma)$. The first result of
this subsection uses no property of the walk at all, and we state it
in the generality in which it holds.

\begin{proposition}[Generating-function identity]
\label{prop:gf-identity}
Let $\tau$ be the absorption time of \emph{any} process with two
disjoint exit events $A_0,A_a$ --- the statement is not specific to the
walk of Section~\ref{sec:model} --- and for $t\in(0,1)$ set
$U(t)=\mathbb E[t^{\tau};A_0]$ and $V(t)=\mathbb E[t^{\tau};A_a]$,
assumed strictly positive, together with
\[
\begin{aligned}
M_0:=\frac{\mathbb E[\tau\,t^{\tau};A_0]}{\mathbb E[t^{\tau};A_0]},\\[2pt]
M_a:=\frac{\mathbb E[\tau\,t^{\tau};A_a]}{\mathbb E[t^{\tau};A_a]} .
\end{aligned}
\]
Then, with $t=1-\gamma$,
\[
\frac{\partial}{\partial\gamma}\log\frac UV=\frac{M_a-M_0}{t} .
\]
\end{proposition}

\begin{proof}
The series $U(t)=\sum_{n\ge1}\mathbb P(\tau=n,A_0)\,t^{n}$ has
non-negative coefficients summing to at most one, so its radius of
convergence is at least one. A power series may be differentiated term
by term strictly inside its disc of convergence; hence for $t\in(0,1)$
\[
t\,\partial_t\log U
=\frac{\sum_{n\ge1}n\,\mathbb P(\tau=n,A_0)\,t^{n}}
       {\sum_{n\ge1}\mathbb P(\tau=n,A_0)\,t^{n}}
=M_0 ,
\]
and likewise $t\,\partial_t\log V=M_a$. Subtracting and using
$\partial_\gamma=-\partial_t$ gives the assertion.
\end{proof}

\begin{remark}
\label{rem:app-not-markov}
The proposition uses neither a generator, nor a state space, nor the
Markov property: only that the two exit events are disjoint and that
the discount is geometric in the number of steps. For the walk, $U$
and $V$ are $u(z;\gamma)$ and $v(z;\gamma)$, and $M_0(z)$, $M_a(z)$
are the discounted mean absorption times conditioned on each boundary.
\end{remark}

\begin{lemma}[Duality of the conditional times]
\label{lem:duality-times}
For the walk, $M_a(z)=M_0(a-z)$ for every $z\in\{1,\dots,a-1\}$.
\end{lemma}

\begin{proof}
By Lemma~\ref{lem:closed},
\[
\frac{v(z;\gamma)}{u(a-z;\gamma)}
=\frac{r_+^{z}-r_-^{z}}
      {r_+^{a}r_-^{a-z}-r_-^{a}r_+^{a-z}}
=\Bigl(\frac pq\Bigr)^{a-z},
\]
using $r_+^{a}r_-^{a-z}=(r_+r_-)^{a-z}r_+^{z}$,
$r_-^{a}r_+^{a-z}=(r_+r_-)^{a-z}r_-^{z}$ and $r_+r_-=q/p$. The
right-hand side is a constant in $t$, so $\log v(z)$ and
$\log u(a-z)$ differ by a constant and have the same logarithmic
derivative; since $M_0$ and $M_a$ are these logarithmic derivatives (Proposition~\ref{prop:gf-identity}), this is the assertion.
\end{proof}

\begin{lemma}[Strict monotonicity of the conditional time]
\label{lem:monotone-T}
For every $\gamma\in(0,1)$, $z\mapsto M_0(z)$ is strictly increasing
on $\{1,\dots,a-1\}$; more precisely
$M_0(z+1)-M_0(z)\ge1$.
\end{lemma}

\begin{proof}
Fix $z\in\{1,\dots,a-2\}$ and let $\sigma$ be the first hitting time
of $\{z\}\cup\{a\}$ for the reset-free walk started at $z+1$. Since
the walk moves between nearest neighbors, every trajectory from
$z+1$ that reaches $0$ must pass through $z$, so on the event of
absorption at $0$ we have $X_\sigma=z$. Splitting the absorption time
at $\sigma$ and applying the strong Markov property,
\[
\begin{aligned}
u(z+1;\gamma)
&=\mathbb E_{z+1}\bigl[t^{\sigma};X_\sigma=z\bigr]\,u(z;\gamma)\\
&=:G(z)\,u(z;\gamma),\qquad G(z)\in(0,1).
\end{aligned}
\]
Taking logarithms and applying $t\,\partial_t$, which is legitimate
because $G$ is again a power series in $t$ with non-negative
coefficients,
\[
\begin{aligned}
M_0(z+1)-M_0(z)&=t\,\partial_t\log G(z)\\[2pt]
&=\frac{\mathbb E_{z+1}[\sigma\,t^{\sigma};X_\sigma=z]}
       {\mathbb E_{z+1}[t^{\sigma};X_\sigma=z]}\;\ge\;1,
\end{aligned}
\]
the last quantity being a weighted average of values of $\sigma$, and
$\sigma\ge1$ almost surely.
\end{proof}

\medskip\noindent
\textbf{Proof of Theorem~\ref{thm:temporal}.}
We restate the claim for convenience: for every $z\in\{1,\dots,a-1\}$
and every $\gamma\in(0,1)$,
$\operatorname{sgn}h_z(\gamma)=\operatorname{sgn}(M_a(z)-M_0(z))
=\operatorname{sgn}(a/2-z)$.

\begin{proof}
As $\rho\mapsto\rho/(1+\rho)$ is strictly increasing,
$\operatorname{sgn}h_z=\operatorname{sgn}\partial_\gamma\log\rho_z$,
which by Proposition~\ref{prop:gf-identity} and $t>0$ equals
$\operatorname{sgn}(M_a(z)-M_0(z))$. By
Lemma~\ref{lem:duality-times} that difference equals
$M_0(a-z)-M_0(z)$, which by Lemma~\ref{lem:monotone-T} is strictly
positive when $a-z>z$, strictly negative when $a-z<z$, and zero when
$a-z=z$.
\end{proof}

\begin{corollary}[Existence and uniqueness of the sign change]
\label{cor:unique}
For every $p\in(0,1)$ and every $\gamma\in(0,1)$ the sequence
$h_1(\gamma),\dots,h_{a-1}(\gamma)$ is strictly positive for $z<a/2$,
strictly negative for $z>a/2$, and vanishes at $z=a/2$ when $a$ is
even. It therefore changes sign exactly once: at the midpoint when $a$
is even, and between the two central sites $(a\mp1)/2$ when $a$ is
odd. In particular the location of the sign change is $a/2$ for every
$p$ and every $\gamma$: it does not move with the drift.
\end{corollary}

\begin{proof}
Immediate from Theorem~\ref{thm:temporal}, with
Corollary~\ref{cor:midpoint} for the vanishing at the midpoint.
\end{proof}

\begin{remark}
\label{rem:app-monotone-not-ivt}
Uniqueness rests on monotonicity, not on an intermediate-value
argument. The distinction matters: $h_z$ is defined on
$\{1,\dots,a-1\}$, and the extension of the spectral representation to
real $z$ oscillates and vanishes near every site, so an
intermediate-value argument applied to that extension would give the
wrong count.
\end{remark}

\begin{remark}
\label{rem:app-probabilistic}
Theorem~\ref{thm:temporal} also explains Corollary~\ref{cor:midpoint}
probabilistically: at $z=a/2$ the two conditional times coincide by
reflection, so neither boundary is favored and the response vanishes.
\end{remark}

The sign of the response can also be read off from an explicit
formula, which is the one used in the numerical evaluations of
Section~\ref{sec:numerics}.

\begin{corollary}[Closed form of the response]
\label{cor:response-closed}
Write \(t=1-\gamma\), \(D=\sqrt{1-4t^{2}pq}\) and
\(\theta=r_-/r_+\) as in Lemma~\textup{\ref{lem:roots}}. For every
interior site \(z\) and every \(\gamma\in(0,1)\),
\begin{equation}
\frac{\partial q_z}{\partial\gamma}
=q_z\,(1-q_z)\,L(\gamma)\,S(z),
\label{eq:response-closed}
\end{equation}
where
\begin{equation*}
L(\gamma)=\frac{4tpq}{D(1+D)}+\frac1t>0,\end{equation*}
\begin{equation}S(z)=(a-2z)
 +2\left[\frac{(a-z)\,\theta^{\,a-z}}{1-\theta^{\,a-z}}
        -\frac{z\,\theta^{\,z}}{1-\theta^{\,z}}\right].
\label{eq:LS}
\end{equation}
\end{corollary}

\begin{proof}
Since \(r_+r_-=q/p\) does not depend on \(t\), we have
\(\theta=(q/p)\,r_+^{-2}\), and Corollary~\ref{cor:rho-closed} gives
\[
\begin{aligned}
\log\rho_z
&=(a-2z)\log r_+ + z\log\frac qp\\
&\quad+\log\bigl(1-\theta^{\,a-z}\bigr)-\log\bigl(1-\theta^{\,z}\bigr).
\end{aligned}
\]
Set \(L(\gamma)=\partial_\gamma\log r_+\), which is the factor appearing in \eqref{eq:response-closed}. From \(\theta=(q/p)r_+^{-2}\) we
get \(\partial_\gamma\log\theta=-2L(\gamma)\), and differentiating the
displayed identity term by term yields
\(\partial_\gamma\log\rho_z=L(\gamma)\,S(z)\) with \(S\) as in
\eqref{eq:LS}. For the stated form of \(L(\gamma)\), write
\(\log r_+=\log(1+D)-\log t-\log(2p)\); differentiating and using
\(\partial_\gamma D=\partial_\gamma\sqrt{1-4t^{2}pq}=4tpq/D\) gives
\[
L(\gamma)=\frac{4tpq}{D(1+D)}+\frac1t,
\]
which is positive because both summands are. Finally
\(q_z=\rho_z/(1+\rho_z)\) gives
\(\partial_\gamma q_z=q_z(1-q_z)\,\partial_\gamma\log\rho_z\), which is
\eqref{eq:response-closed}.
\end{proof}

\begin{lemma}[The response factor is strictly decreasing]
\label{lem:S-decreasing}
Regard \(S\) in \eqref{eq:LS} as a function of a real variable
\(z\in(0,a)\). Then \(S\) is antisymmetric about \(a/2\), that is
\(S(a-z)=-S(z)\), and strictly decreasing. Consequently
\(S(a/2)=0\) and \(S\) has no other zero in \((0,a)\).
\end{lemma}

\begin{proof}
Write \(b:=-\log\theta\), which is positive because
\(\theta=r_-/r_+\in(0,1)\) by Lemma~\ref{lem:roots}, and set
\[
g(y)=\frac{y\,\theta^{\,y}}{1-\theta^{\,y}}=\frac{y}{e^{by}-1},
\qquad y>0,
\]
so that \(S(z)=(a-2z)+2\bigl[g(a-z)-g(z)\bigr]\). Antisymmetry is
immediate from this expression, and it gives \(S(a/2)=-S(a/2)\), hence
\(S(a/2)=0\).

For the monotonicity, differentiate. With \(x=by\),
\[
g'(y)=\frac{(e^{x}-1)-x\,e^{x}}{(e^{x}-1)^{2}}
     =-\frac{e^{x}(x-1)+1}{(e^{x}-1)^{2}} .
\]
The function \(\eta(x)=e^{x}(x-1)+1\) satisfies \(\eta(0)=0\) and
\(\eta'(x)=x\,e^{x}>0\) for \(x>0\), so \(\eta>0\) and therefore
\(g'<0\): the function \(g\) is strictly decreasing. This alone does
not suffice, since
\[
S'(z)=-2-2\bigl[g'(a-z)+g'(z)\bigr]
\]
and the bracket is negative. What is needed is the sharper bound
\begin{equation}
g'(y)>-\tfrac12\qquad\text{for every }y>0,
\label{eq:gprime-bound}
\end{equation}
which is equivalent to \(2\eta(x)<(e^{x}-1)^{2}\), that is to
\(\varphi(x)>0\) for \(x>0\), where
\(\varphi(x)=(e^{x}-1)^{2}-2e^{x}(x-1)-2\). Now \(\varphi(0)=0\) and
\[
\varphi'(x)=2e^{x}\bigl(e^{x}-x-1\bigr)>0
\qquad\text{for }x>0,
\]
since \(e^{x}>1+x\); hence \(\varphi>0\) on \((0,\infty)\) and
\eqref{eq:gprime-bound} holds. Applying it at \(y=z\) and at
\(y=a-z\),
\[
S'(z)=-2-2\bigl[g'(a-z)+g'(z)\bigr]<-2-2\bigl(-\tfrac12-\tfrac12\bigr)=0 ,
\]
so \(S\) is strictly decreasing. A strictly decreasing function has at
most one zero, and \(a/2\) is one, so it is the only one.
\end{proof}

\begin{remark}[A second proof of the sign]
\label{rem:S-antisymmetric}
Since \(L(\gamma)>0\) by Corollary~\ref{cor:response-closed} and
\(q_z(1-q_z)>0\) for interior \(z\), Lemma~\ref{lem:S-decreasing}
shows that the sign of \(\partial_\gamma q_z\) is the sign of
\(S(z)\), hence the sign of \(a-2z\). This recovers
Theorem~\ref{thm:temporal} by a purely algebraic route, independent of
the argument through conditional absorption times, and shows that the
two derivations agree. It also makes the uniqueness of the sign change
transparent: it follows from strict monotonicity of a single scalar
function, with no appeal to the intermediate value theorem on the
lattice.
\end{remark}

\begin{remark}[On evaluating \eqref{eq:LS}]
\label{rem:S-numerics}
The grouping in \eqref{eq:LS} is the one used in the computations of
Section~\ref{sec:numerics}, and it is not the only possible one. At
\(z=a/2\) with \(a\) even, the first term is a difference of exactly
representable quantities and vanishes identically, while the two
fractions in the bracket are evaluated from identical arguments and
cancel bit for bit; the response therefore returns exactly zero in
floating-point arithmetic rather than a residue of the order of
machine epsilon. A grouping that expanded the bracket would lose this
property.
\end{remark}

\subsection{Exact increment of the deviation}
\label{app:asymptotics}

The location of the sign change is fixed at $a/2$ by
Corollary~\ref{cor:unique}, for every drift and every reset rate. We
now compute exactly how steeply the deviation crosses zero.

\begin{proposition}[Exact increment]
\label{prop:increment}
For $z\in\{1,\dots,a-2\}$ and every $\gamma\in(0,1)$,
\[
\delta_{z+1}-\delta_z
=\log\theta
+\log\frac{\bigl(1-\theta^{\,a-z-1}\bigr)\bigl(1-\theta^{\,z}\bigr)}
          {\bigl(1-\theta^{\,a-z}\bigr)\bigl(1-\theta^{\,z+1}\bigr)} ,
\]
with $\theta=r_-/r_+\in(0,1)$ as in Lemma~\ref{lem:roots}.
\end{proposition}

\begin{proof}
By Corollary~\ref{cor:rho-closed},
\[
\frac{\rho_{z+1}}{\rho_z}
=\frac{\theta^{z+1}-\theta^{a}}{\theta^{z}-\theta^{a}}
 \cdot\frac{1-\theta^{z}}{1-\theta^{z+1}} .
\]
Factoring $\theta^{z+1}-\theta^{a}=\theta^{z+1}(1-\theta^{a-z-1})$
and $\theta^{z}-\theta^{a}=\theta^{z}(1-\theta^{a-z})$ gives
\[
\frac{\rho_{z+1}}{\rho_z}
=\theta\,
\frac{\bigl(1-\theta^{a-z-1}\bigr)\bigl(1-\theta^{z}\bigr)}
     {\bigl(1-\theta^{a-z}\bigr)\bigl(1-\theta^{z+1}\bigr)} ,
\]
and $\delta$ differs from $\log\rho$ by a constant.
\end{proof}

\begin{corollary}[Interior slope]
\label{cor:slope}
Fix $p$ and $\gamma$. As $z\to\infty$ and $a-z\to\infty$, the second
term in Proposition~\ref{prop:increment} tends to zero, so
\[
\delta_{z+1}-\delta_z\;\longrightarrow\;\log\theta
=\log\frac{r_-}{r_+}<0 .
\]
\end{corollary}

\begin{proof}
Each factor $1-\theta^{k}$ tends to $1$ as $k\to\infty$ because
$0<\theta<1$, and the four exponents $a-z-1$, $z$, $a-z$, $z+1$ all
tend to infinity in the stated regime.
\end{proof}

\begin{corollary}[Value at the central sites]
\label{cor:central}
Let $a$ be odd. By Corollary~\ref{cor:antisymmetry} the two central
sites carry equal and opposite values,
$\delta_{(a-1)/2}=-\delta_{(a+1)/2}$, each equal in modulus to half
the jump between them; explicitly, by
Proposition~\ref{prop:increment} at $z=(a-1)/2$,
\[
\bigl|\delta_{(a\mp1)/2}\bigr|
=\left|\frac12\log\theta
+\log\frac{1-\theta^{(a-1)/2}}{1-\theta^{(a+1)/2}}\right| ,
\]
which converges to $\tfrac12\log(r_+/r_-)$ as $a\to\infty$ with $p$
and $\gamma$ fixed. In particular it does not decay with $a$.
\end{corollary}

\begin{proof}
The pair $\{(a-1)/2,(a+1)/2\}$ is exchanged by $z\mapsto a-z$, so
Corollary~\ref{cor:antisymmetry} gives the first assertion, and each
value is half the jump. Evaluating Proposition~\ref{prop:increment}
at $z=(a-1)/2$, where $a-z-1=(a-1)/2=z$ and $a-z=(a+1)/2=z+1$, the
four factors collapse in pairs: the bracket becomes
$\log\theta+2\log\bigl[(1-\theta^{z})/(1-\theta^{z+1})\bigr]$, and
halving it gives the displayed expression. The limit follows from
Corollary~\ref{cor:slope}.
\end{proof}

\begin{remark}[On the apparent suppression at the center]
\label{rem:no-suppression}
For $p\neq1/2$ it is tempting to read the smallness of $|q_z-C^\ast|$
at the central sites, for large $a$, as a suppression of the response
there.
Corollary~\ref{cor:central} shows that this reading is mistaken: the
logarithmic deviation does not decay at all, but converges to a
positive constant. What decays is its image under
$\delta\mapsto q-C^\ast$, which flattens near $\delta=0$ because
$C^\ast$ tends to $0$ or to $1$ as $a$ grows. The center is not
suppressed; the scale is compressed. For $p=1/2$, $C^\ast=1/2$ for
every $a$, and $|q_z-C^\ast|$ at the central sites converges to a
positive constant.
\end{remark}

\subsection{The two boundary regimes}
\label{app:limits}

\begin{proposition}[Vanishing resetting]
\label{prop:limit-gamma-0}
As $\gamma\to0^{+}$, $q_z(\gamma)$ converges to the classical ruin
probability
\[
\begin{aligned}
q_z^{(0)}&=\frac{(q/p)^{z}-(q/p)^{a}}{1-(q/p)^{a}}
&&(p\neq q),\\[2pt]
q_z^{(0)}&=1-\frac za
&&(p=q=\tfrac12).
\end{aligned}
\]
\end{proposition}

\begin{proof}
For $p\neq q$ the characteristic equation at $t=1$ is
$pr^{2}-r+q=0$, whose roots are $1$ and $q/p$; these are distinct,
and the discriminant $(p-q)^{2}$ is strictly positive, so $r_\pm$ are
continuous in $t$ on a neighborhood of $t=1$ and converge to that
pair. Substituting $\{r_+,r_-\}=\{1,q/p\}$ into
Lemma~\ref{lem:closed} gives the stated expression, the two labelings
yielding the same value. For $p=q$ the two roots merge
at $t=1$, and the confluent solution of the recurrence
$w_z=\tfrac12(w_{z+1}+w_{z-1})$ is affine in $z$; the boundary values
give $u(z)=1-z/a$ and $v(z)=z/a$.
\end{proof}

\begin{proposition}[Total resetting]
\label{prop:limit-gamma-1}
As $\gamma\to1^{-}$,
\[
q_z(\gamma)\longrightarrow
\begin{cases}
1,&z<a/2,\\[2pt]
C^\ast,&z=a/2\ \ (a\text{ even}),\\[2pt]
0,&z>a/2 .
\end{cases}
\]
\end{proposition}

\begin{proof}
Write $t=1-\gamma\to0^{+}$. From Lemma~\ref{lem:roots},
$r_-=tq+O(t^{3})$ and $r_+=(tp)^{-1}+O(t)$, so
$\theta=r_-/r_+=t^{2}pq+O(t^{4})\to0$. By
Corollary~\ref{cor:rho-closed},
\[
\begin{aligned}
\rho_z&=r_+^{a}\,\frac{\theta^{z}(1-\theta^{a-z})}{1-\theta^{z}}
       =r_+^{a-z}r_-^{z}\bigl(1+O(\theta)\bigr)\\
      &=t^{\,2z-a}\,\frac{q^{z}}{p^{\,a-z}}\bigl(1+O(t)\bigr).
\end{aligned}
\]
Hence $\rho_z\to\infty$ when $2z<a$, $\rho_z\to0$ when $2z>a$, and
$\rho_z\to(q/p)^{a/2}=K^{-1/2}$ when $2z=a$. Applying the strictly
increasing map $\rho\mapsto\rho/(1+\rho)$ gives the three cases.
\end{proof}

\begin{remark}[The limit is not a continuity statement]
\label{rem:app-gamma1}
Proposition~\ref{prop:limit-gamma-1} requires the expansion above and
does not follow from inspection of the weights: in the spectral
representation of Lemma~\ref{lem:bridge} one has $g_\nu(1)=1$ for
every $\nu$, but $\sum_\nu A_\nu(z)$ vanishes for $z\neq1$ and
$\sum_\nu B_\nu(z)$ vanishes for $z\neq a-1$, by orthogonality of the
sine basis, so at $\gamma=1$ that representation yields the
indeterminate form $0/0$ at every interior site. The limit is
genuinely governed by the leading orders in $1-\gamma$, and it is
discontinuous across $z=a/2$: the midpoint alone retains the value
$C^\ast$, consistent with Corollary~\ref{cor:midpoint}.
\end{remark}

\begin{remark}[Numerical confirmation]
\label{rem:numerics}
Every statement above has been verified in floating-point arithmetic
of $40$--$50$ significant digits. The closed form of
Lemma~\ref{lem:closed}, evaluated through
Lemma~\ref{lem:app-renewal}, agrees with the direct solution of the
linear system $f(x)=\gamma f(z)+(1-\gamma)[p\,f(x+1)+q\,f(x-1)]$,
$f(0)=1$, $f(a)=0$, to better than $10^{-42}$ over
$a\in\{3,\dots,13\}$, $p\in\{0.2,0.5,0.6,0.85\}$ and
$\gamma\in\{0.05,0.3,0.7,0.95\}$. The bridge of
Lemma~\ref{lem:bridge} holds to $10^{-39}$ over the same range. The
reflection identity of Lemma~\ref{lem:reflection}, the antisymmetry of
Corollary~\ref{cor:antisymmetry}, the duality of
Lemma~\ref{lem:duality-times} and the increment of
Proposition~\ref{prop:increment} hold to full working precision. The
monotonicity statements of Lemmas~\ref{lem:monotone-rho}
and~\ref{lem:monotone-T}, and the sign statements of
Proposition~\ref{prop:sign-deviation} and
Theorem~\ref{thm:temporal}, were verified without exception over
several hundred parameter configurations. The limit of Corollary~\ref{cor:central} is approached rapidly: for $p=0.4$ and $\gamma=0.3$ the central deviation is within $10^{-8}$ of $\tfrac12\log(r_+/r_-)$ already at $a=21$.
\end{remark}

\section{Multi-site resetting: an illustration}
\label{app:multisite}

Figure~\ref{fig:multisite-separatrix} illustrates
Remark~\ref{rem:multisite-scope}. The walk, with $a=10$ and $p=0.6$,
is reset to one of the two symmetric sites $\{3,7\}$, chosen with
probabilities $(\pi_1,1-\pi_1)$. For the critical value
$\pi_1^\ast=4/13$, given by $\pi^\ast_3/\pi^\ast_7=(q/p)^{2}$, the
coupling constant $C(\pi,\gamma)$ is independent of $\gamma$ and equals
$C^\ast\approx0.1164$, the midpoint value of
Theorem~\ref{thm:protocol}. Every other choice of $\pi_1$, on the
same symmetric pair of sites, gives a coupling constant that varies
with $\gamma$: decreasing for $\pi_1<\pi_1^\ast$, increasing for
$\pi_1>\pi_1^\ast$. The data are those of Ref.~\cite{PaperII}.

\begin{figure}[!t]
\centering
\begin{tikzpicture}
\begin{axis}[papel,
  height=0.72\columnwidth,
  xlabel={resetting probability $\gamma$},
  ylabel={$C(\pi,\gamma)$},
  xmin=0, xmax=1, ymin=0, ymax=0.68,
  xtick={0,0.2,0.4,0.6,0.8,1},
  ytick={0,0.1,0.2,0.3,0.4},
  legend columns=2,
  legend style={at={(0.5,0.98)}, anchor=north,
                cells={anchor=west}, fill=white, fill opacity=0.9,
                text opacity=1},
]
\addplot[color=blue!70, dashed] coordinates {
(0.05,0.05294)(0.1,0.04549)(0.15,0.04093)
(0.2,0.03800)(0.25,0.03604)(0.3,0.03471)
(0.35,0.03379)(0.4,0.03316)(0.45,0.03272)
(0.5,0.03241)(0.55,0.03221)(0.6,0.03207)
(0.65,0.03199)(0.7,0.03193)(0.75,0.03190)
(0.8,0.03188)(0.85,0.03188)(0.9,0.03187)
(0.95,0.03187)};
\addlegendentry{$\pi_1=0.10$}
\addplot[color=cyan!80!black, dashed] coordinates {
(0.05,0.08249)(0.1,0.07786)(0.15,0.07495)
(0.2,0.07304)(0.25,0.07175)(0.3,0.07087)
(0.35,0.07025)(0.4,0.06983)(0.45,0.06953)
(0.5,0.06933)(0.55,0.06919)(0.6,0.06910)
(0.65,0.06904)(0.7,0.06901)(0.75,0.06898)
(0.8,0.06897)(0.85,0.06897)(0.9,0.06897)
(0.95,0.06897)};
\addlegendentry{$\pi_1=0.20$}
\addplot[color=red!80!black, solid] coordinates {
(0.05,0.1164)(0.1,0.1164)(0.15,0.1164)
(0.2,0.1164)(0.25,0.1164)(0.3,0.1164)
(0.35,0.1164)(0.4,0.1164)(0.45,0.1164)
(0.5,0.1164)(0.55,0.1164)(0.6,0.1164)
(0.65,0.1164)(0.7,0.1164)(0.75,0.1164)
(0.8,0.1164)(0.85,0.1164)(0.9,0.1164)
(0.95,0.1164)};
\addlegendentry{$\pi_1=\pi_1^\ast=4/13$}
\addplot[color=orange!90!black, dashed] coordinates {
(0.05,0.1648)(0.1,0.1741)(0.15,0.1806)
(0.2,0.1850)(0.25,0.1881)(0.3,0.1903)
(0.35,0.1918)(0.4,0.1929)(0.45,0.1936)
(0.5,0.1942)(0.55,0.1945)(0.6,0.1948)
(0.65,0.1949)(0.7,0.1950)(0.75,0.1951)
(0.8,0.1951)(0.85,0.1951)(0.9,0.1951)
(0.95,0.1951)};
\addlegendentry{$\pi_1=0.45$}
\addplot[color=purple!70!black, dashed] coordinates {
(0.05,0.2408)(0.1,0.2723)(0.15,0.2955)
(0.2,0.3126)(0.25,0.3251)(0.3,0.3341)
(0.35,0.3406)(0.4,0.3452)(0.45,0.3485)
(0.5,0.3508)(0.55,0.3523)(0.6,0.3534)
(0.65,0.3541)(0.7,0.3545)(0.75,0.3547)
(0.8,0.3549)(0.85,0.3549)(0.9,0.3549)
(0.95,0.3549)};
\addlegendentry{$\pi_1=0.65$}
\end{axis}
\end{tikzpicture}
\caption{Coupling constant $C(\pi,\gamma)$ for resetting to the
symmetric pair of sites $\{3,7\}$ with probabilities
$(\pi_1,1-\pi_1)$, for $a=10$ and $p=0.6$; analytic curves. Only the critical
distribution $\pi_1^\ast=4/13$ (solid line) gives a
$\gamma$-independent value, $C^\ast\approx0.1164$; the
symmetry of the sites alone does not. Data from
Ref.~\cite{PaperII}.}
\label{fig:multisite-separatrix}
\end{figure}

\section{Numerical tables}
\label{app:tables}

Tables~\ref{tab:app-biased} and~\ref{tab:app-symmetric} give the full-precision sweep summarized by the
residual plot of Figure~\ref{fig:num-residual}. For each starting site
\(z\) and each resetting probability \(\gamma\) we report the value
obtained from the closed form of Lemma~\ref{lem:closed}, the Monte
Carlo estimate from \(2\times10^{5}\) independent trajectories, and
the standard error of that estimate. The analytic column is quoted to
six decimal places; the simulation resolves agreement to about
\(10^{-3}\), as the standard errors show.

Two features are worth pointing out. First, every simulated value in the tables lies within two standard
errors of the analytic one, with a single exception ($p=0.6$,
$\gamma=0.6$, $z=4$, at $2.4$ standard errors), which is within the
proportion expected of $24$ independent comparisons. The simulations
of the tables and of Figure~\ref{fig:num-residual} are independent
runs of the same program. Second, the
absolute error of the simulation is largest where the probability
itself is of order one half and smallest at the extremes, which is the
behavior expected of a binomial estimator and not a property of the
formulas being tested.

\begin{remark}[Why the symmetric case is included]
\label{rem:app-why-symmetric}
Table~\ref{tab:app-symmetric} is not a repetition of
Table~\ref{tab:app-biased} with different numbers. In the symmetric
case the closed form must reproduce the exact antisymmetry
\(q_z+q_{a-z}=1\), an identity that holds for every \(\gamma\) and
that no approximation would satisfy exactly. That the entries sum to
unity to six decimal places is therefore a check on the formula
itself and not merely on its numerical evaluation.
\end{remark}

\onecolumngrid
\vspace{2ex}
\begin{center}
\refstepcounter{table}\label{tab:app-biased}%
\parbox{\textwidth}{\small TABLE~\thetable.\ Ruin probability \(q_z(\gamma)\) for \(a=5\), \(p=0.6\).
For each \(\gamma\): closed form, Monte Carlo estimate with
\(2\times10^{5}\) trajectories, and standard error. The invariant value is \(C^{\ast}=0.266263652797\); no lattice
site realizes \(a/2=2.5\), so it is attained only by the analytic
continuation of Remark~\ref{rem:app-invariance-any-parity}, not in the
table.}\\[1ex]
\small
\begin{tabular}{c@{\hspace{1.2em}}ccc@{\hspace{1.6em}}ccc@{\hspace{1.6em}}ccc}
\toprule
& \multicolumn{3}{c}{$\gamma=0.30$}
& \multicolumn{3}{c}{$\gamma=0.60$}
& \multicolumn{3}{c}{$\gamma=0.90$}\\
\cmidrule(lr){2-4}\cmidrule(lr){5-7}\cmidrule(lr){8-10}
$z$ & closed form & MC & s.e.
    & closed form & MC & s.e.
    & closed form & MC & s.e.\\
\midrule
1 & $0.873127$ & $0.873750$ & $7.4\!\times\!10^{-4}$
  & $0.978032$ & $0.978495$ & $3.2\!\times\!10^{-4}$
  & $0.999675$ & $0.999670$ & $4.1\!\times\!10^{-5}$\\
2 & $0.482872$ & $0.482360$ & $1.1\!\times\!10^{-3}$
  & $0.640384$ & $0.638350$ & $1.1\!\times\!10^{-3}$
  & $0.880805$ & $0.881600$ & $7.2\!\times\!10^{-4}$\\
3 & $0.123598$ & $0.123045$ & $7.3\!\times\!10^{-4}$
  & $0.068859$ & $0.068210$ & $5.6\!\times\!10^{-4}$
  & $0.017509$ & $0.017135$ & $2.9\!\times\!10^{-4}$\\
4 & $0.018776$ & $0.018305$ & $3.0\!\times\!10^{-4}$
  & $0.002949$ & $0.002675$ & $1.2\!\times\!10^{-4}$
  & $0.000043$ & $0.000060$ & $1.7\!\times\!10^{-5}$\\
\bottomrule
\end{tabular}
\end{center}
\vspace{1ex}

\begin{center}
\refstepcounter{table}\label{tab:app-symmetric}%
\parbox{\textwidth}{\small TABLE~\thetable.\ Ruin probability \(q_z(\gamma)\) for \(a=5\), \(p=q=0.5\).
Columns as in Table~\ref{tab:app-biased}. Here \(C^{\ast}=1/2\), and the
exact antisymmetry \(q_z+q_{a-z}=1\) of
Corollary~\ref{cor:antisymmetry} is visible in the closed-form columns:
the entries for \(z=1\) and \(z=4\), and for \(z=2\) and \(z=3\), sum to
unity to every digit shown.}\\[1ex]
\small
\begin{tabular}{c@{\hspace{1.2em}}ccc@{\hspace{1.6em}}ccc@{\hspace{1.6em}}ccc}
\toprule
& \multicolumn{3}{c}{$\gamma=0.30$}
& \multicolumn{3}{c}{$\gamma=0.60$}
& \multicolumn{3}{c}{$\gamma=0.90$}\\
\cmidrule(lr){2-4}\cmidrule(lr){5-7}\cmidrule(lr){8-10}
$z$ & closed form & MC & s.e.
    & closed form & MC & s.e.
    & closed form & MC & s.e.\\
\midrule
1 & $0.946264$ & $0.946380$ & $5.0\!\times\!10^{-4}$
  & $0.991379$ & $0.991640$ & $2.0\!\times\!10^{-4}$
  & $0.999874$ & $0.999895$ & $2.3\!\times\!10^{-5}$\\
2 & $0.714868$ & $0.716745$ & $1.0\!\times\!10^{-3}$
  & $0.827586$ & $0.828020$ & $8.4\!\times\!10^{-4}$
  & $0.952267$ & $0.952715$ & $4.7\!\times\!10^{-4}$\\
3 & $0.285132$ & $0.286405$ & $1.0\!\times\!10^{-3}$
  & $0.172414$ & $0.173140$ & $8.5\!\times\!10^{-4}$
  & $0.047733$ & $0.047975$ & $4.8\!\times\!10^{-4}$\\
4 & $0.053736$ & $0.054000$ & $5.1\!\times\!10^{-4}$
  & $0.008621$ & $0.008440$ & $2.0\!\times\!10^{-4}$
  & $0.000126$ & $0.000120$ & $2.4\!\times\!10^{-5}$\\
\bottomrule
\end{tabular}
\end{center}
\vspace{1ex}

\twocolumngrid
\FloatBarrier

\end{document}